\documentclass[11pt, a4paper]{amsart}

\usepackage{amsmath, amssymb, amsfonts, amsthm}
\usepackage{mathrsfs}
\usepackage{verbatim}
\usepackage[backend=biber]{biblatex}
\renewbibmacro{in:}{}

\usepackage[dvipsnames]{xcolor}
\usepackage{graphicx}
\usepackage{hyperref}
\usepackage{tikz}
\usetikzlibrary{cd}
\tikzset{
node distance=1.5cm,
la/.style={scale=0.8}
}

\usepackage{geometry}
\usepackage{appendix}

\newcommand{\OSA}{{\overline{\mathscr{A}}}}

\newcommand{\SA}{{\mathscr{A}}}
\newcommand{\SB}{{\mathscr{B}}}

\newcommand{\SE}{{\mathscr{E}}}
\newcommand{\SF}{{\mathscr{F}}}
\newcommand{\SG}{{\mathscr{G}}}

\newcommand{\SI}{{\mathscr{I}}}

\newcommand{\SL}{{\mathscr{L}}}

\newcommand{\SP}{{\mathscr{P}}}

\newcommand{\SY}{{\mathscr{Y}}}

\newcommand{\BC}{{\mathbb{C}}}

\newcommand{\BN}{{\mathbb{N}}}

\newcommand{\BP}{{\mathbb{P}}}
\newcommand{\BQ}{{\mathbb{Q}}}
\newcommand{\BR}{{\mathbb{R}}}

\renewcommand{\phi}{\varphi}

\newcommand{\Res}{\operatorname{Res}}
\newcommand{\im}{\operatorname{im}}
\newcommand{\codim}{\operatorname{codim}}
\newcommand{\coker}{\operatorname{coker}}
\newcommand{\Gr}{\operatorname{Gr}}
\newcommand{\can}{\operatorname{can}}

\theoremstyle{plain}
\newtheorem{thm}{Theorem}[section]

\newtheorem{lemma}[thm]{Lemma}
\newtheorem{prop}[thm]{Proposition}
\newtheorem{cor}[thm]{Corollary}

\newtheorem{theorem}[thm]{Theorem}

\theoremstyle{definition}
\newtheorem{definition}[thm]{Definition}
\newtheorem{example}[thm]{Example}
\newtheorem{remark}[thm]{Remark}

\usepackage[textwidth=26mm, textsize = footnotesize]{todonotes}
\title{Affine Hyperplane Arrangements at Finite Distance}

\author{Anaëlle Pfister}
\address{
Max Planck Institute for Mathematics in the Sciences, Leipzig, Germany}
\email{anaelle.pfister@mis.mpg.de}
\date{\today}

\begin{document}

\begin{abstract}
We study the relative homology group of an affine hyperplane arrangement and its Poincaré dual, the cohomology at finite distance of the complement. We give an Orlik--Solomon-type description of the latter, and identify it with the vector space of logarithmic forms having vanishing residues at infinity. To this end, we introduce a partial version of wonderful compactifications, which could be relevant in other contexts where blow-ups only occur at infinity. Finally, we show that the cohomology at finite distance coincides with the vector space of canonical forms in the sense of positive geometry.
\end{abstract}

\maketitle

\section{Introduction}
Let us consider a collection of distinct hyperplanes $\SA$ in $\BC^n,$ which we call an affine hyperplane arrangement. A central object in the study of such arrangements is the cohomology of the complement $\BC^n \setminus \SA$ following the foundational work of Arnol'd \cite{arnold2013cohomology}, Brieskorn \cite{brieskorn2006groupes}, and Orlik--Solomon \cite{Orlik-Solomon}.
Another natural cohomological invariant of the pair $(\BC^n, \SA),$ which, to the best of our knowledge, has not yet been studied, is the cohomology group
$$H^k_{< \infty}(\BC^n \setminus \SA):=H^k_c(\BC^n, Rj_* \BQ_{\BC^n \setminus \SA}),$$
where $j: \BC^n \setminus \SA \hookrightarrow \BC^n$ denotes the open immersion. We refer to it as \emph{cohomology group at finite distance of $\BC^n \setminus \SA$} (with rational coefficients).
This terminology is motivated by the following heuristic interpretation: if one temporarily replaces the constant sheaf $\BQ_{\BC^n \setminus \SA}$ by $\BR_{\BC^n \setminus \SA},$ then $H^n_{< \infty}(\BC^n \setminus \SA)$ can be computed using the complex of smooth real $n$-forms $\omega$ on $\BC^n \setminus \SA$ for which there exists a compact $K \subseteq \BC^n$ such that $\omega$ vanishes outside $K$. This is notably different from the compactly supported cohomology of $\BC^n \setminus \SA$ where the compact is required to live inside $\BC^n \setminus \SA.$
This intuitive description is made precise in Appendix \ref{section3.0}.

The cohomology at finite distance is isomorphic to the relative homology by Poincaré duality:
\begin{equation}
	\label{Poincare1}
	H^k_{< \infty}(\BC^n \setminus \SA) \simeq H_{2n-k}(\BC^n, \SA).
\end{equation}

\begin{remark}
Note that for $\SP \subseteq \BP_\BC^n,$ a projective hyperplane arrangement, we have the following isomorphisms:
$$H^k_c(\BP^n_\BC, Rj'_* \BQ_{\BP^n_\BC \setminus \SP})\simeq H^k_c(\BP^n_\BC \setminus \SP) \simeq H^k(\BP^n_\BC \setminus \SP) \simeq H_{2n-k}(\BP^n_\BC, \SP),$$
where $j': \BP^n_\BC \setminus \SP \to \BP^n_\BC$ denotes the open immersion. The first two isomorphisms follow from the compactness of $\BP^n_\BC$, while the last one is given by Poincaré duality. Thus, in the projective setting, the cohomology at finite distance does not produce additional information beyond the usual cohomology of the complement.
\end{remark}

The first result of this article is an Orlik--Solomon-type description of the cohomology at finite distance of the complement.
For an affine hyperplane arrangement $\SA$, the Orlik--Solomon algebra $A_\bullet(\SA)$ provides a combinatorial description of the cohomology of the complement.
This algebra is defined purely combinatorially and yields a canonical isomorphism: $$H^\bullet(\BC^n \setminus \SA) \simeq A_\bullet(\SA).$$
For our description, we also need the maps $\partial: A_\bullet(\SA) \to A_{\bullet-1}(\SA)$ making $A_\bullet(\SA)$ a complex. Note that the map $\partial: A_\bullet(\SA) \to A_{\bullet-1}(\SA)$ is not a derivation for non-central arrangements.

\begin{theorem}[\ref{thmcharacterization} and \ref{propchar}]
\label{char_intro}
Let $\SA \subseteq \BC^n$ be an essential affine hyperplane arrangement.
Then, the only non-vanishing cohomology group is the middle one, i.e., $$H^k_{< \infty}(\BC^n \setminus \SA)=0 \text{ if } k \ne n.$$
Furthermore, $$H^n_{< \infty}(\BC^n \setminus \SA) \simeq \ker(\partial: A_n(\SA) \to A_{n-1}(\SA)).$$
\end{theorem}

A consequence of Orlik--Solomon's result is the following. The Betti numbers of $\BC^n \setminus \SA$ are encoded in an invariant of the arrangement, namely the characteristic polynomial $\chi_\SA(t)$ (see also \cite[Theorem 3.68]{Orlik-Terao}). Similarly, in the present setting, Theorem \ref{char_intro} implies that the dimension of $H^n_{< \infty}(\BC^n \setminus \SA)$ is similarly captured by the characteristic polynomial as follows (see Proposition \ref{propchar}):
\begin{equation}
	\label{dim}
	\dim(H^n_{< \infty}(\BC^n \setminus \SA))\dim(H_n(\BC^n, \SA))=(-1)^n\chi_{\SA}(1).
\end{equation}
When $\SA$ is the complexification of a real arrangement $\SA_\BR$ in $\BR^n$, this dimension coincides with the number of bounded regions of $\SA_\BR$ by Zaslavsky's theorem (see \cite[Theorem C]{zaslavsky1997facing}).
In this article, we prove that the bounded regions of the real arrangement $\SA_\BR$ form a basis of the relative homology group $H_n(\BC^n, \SA)$ (see Proposition \ref{propboundedreg}), in agreement with \eqref{dim}.

We illustrate these results in the following example.

\begin{example}
\label{runningexample}
Consider the hyperplane arrangement $\SA=\{L_1,L_2,L_3,L_4,L_5\} \subseteq \BC^2$ where $$L_1=\{x_1=0\}, \ L_2=\{x_1-1=0\}, \ L_3=\{x_2=0\}, L_4=\{x_2-1=0\} \text{ and } L_5=\{x_1-x_2=0\}$$ as illustrated in Figure \ref{figrun}. The homogeneous components of the Orlik--Solomon algebra are given by
$$A_1(\SA)= \BQ e_1 \oplus \BQ e_2 \oplus \BQ e_3 \oplus \BQ e_4 \oplus \BQ e_5$$
and $$ A_2(\SA)= \frac{\BQ e_1 \wedge e_3 \oplus \BQ e_1 \wedge e_4 \oplus \BQ e_1 \wedge e_5 \oplus \BQ e_2 \wedge e_3 \oplus \BQ e_2 \wedge e_4 \oplus \BQ e_2 \wedge e_5 \oplus \BQ e_3 \wedge e_5 \oplus \BQ e_4 \wedge e_5}{(e_1 \wedge e_3-e_1 \wedge e_5+e_3 \wedge e_5, e_2 \wedge e_4-e_2 \wedge e_5+e_4 \wedge e_5)}.$$ The differential $\partial: A_2(\SA) \to A_1(\SA)$ is defined on generators by $\partial(e_i \wedge e_j)=e_j-e_i,$ where $e_i \wedge e_j$ is a generator of $A_2(\SA).$
Therefore, the cohomology at finite distance of the complement satisfies: $$H^2_{< \infty}(\BC^n \setminus \SA) \simeq \ker(\partial)=\BQ(e_1 \wedge e_4- e_1 \wedge e_5 +e_4 \wedge e_5) \oplus \BQ (e_2 \wedge e_3- e_2 \wedge e_5 +e_3 \wedge e_5).$$ In particular, its dimension is 2, which coincides with the number of bounded regions. Moreover, each generator can be interpreted as a sum of differential forms associated with the vertices of one of the two bounded regions. This phenomenon is a consequence of Poincaré duality \eqref{Poincare1}, assigning to each bounded region, a class of forms in $H^n_{< \infty}(\BC^n \setminus \SA)$ and will be made precise in Section \ref{Section4} in the discussion about canonical forms.

\begin{figure}

\tikzset{every picture/.style={line width=0.75pt}}

\begin{tikzpicture}[x=0.75pt,y=0.75pt,yscale=-1,xscale=1]

\draw    (51,41.67) -- (51,226.67) ;
 
\draw    (121,41.67) -- (121,226.67) ;

\draw    (166,181) -- (5,180) ;
 
\draw    (169,111) -- (8,110) ;

\draw    (17.33,214.33) -- (190.33,40.67) ;

\draw (41,234) node [anchor=north west][inner sep=0.75pt]   [align=left] {$L_{1}$};

\draw (115,234) node [anchor=north west][inner sep=0.75pt]   [align=left] {$L_{2}$};

\draw (179,171) node [anchor=north west][inner sep=0.75pt]   [align=left] {$L_{3}$};

\draw (185,105) node [anchor=north west][inner sep=0.75pt]   [align=left] {$L_{4}$};

\draw (200,27) node [anchor=north west][inner sep=0.75pt]   [align=left] {$L_{5}$};

\end{tikzpicture}
\caption{The arrangement $\SA.$}
\label{figrun}
\end{figure}
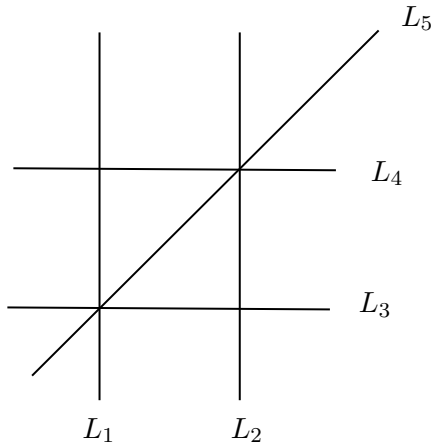
\end{example}

A direct consequence of Theorem \ref{char_intro} is the following inclusion:
$$H^n_{< \infty}(\BC^n \setminus \SA) \hookrightarrow H^n(\BC^n \setminus \SA).$$ In Section \ref{Section3}, we characterize the image of this map. The main result of this section is an interpretation of the cohomology at finite distance in terms of residues at infinity.
We present two different such descriptions.
Let $X$ be a compactification of $\BC^n$, and denote by $Y:= X \setminus \BC^n$ the divisor at infinity and by $\OSA$ the closure of $\SA$ in $X$. Here, we do not require the divisor $Y \cup \OSA$ to be normal crossing, but a weaker condition, which we call \emph{being locally a product for $\SA$} (see Definition \ref{locallyaproduct}). 
Under this assumption, there is an isomorphism (see Lemma \ref{lemmainclusion}):
$$H^n_{< \infty}(\BC^n \setminus \SA) \simeq H^n(X \setminus \OSA, Y \setminus (Y \cap \OSA)).$$
Let $(Y_i)_{i \in I}$ denote the irreducible components of $Y$ and set $$Y_i^\circ=Y_i \setminus \left(\bigcup_{i \ne j \in I}Y_i \cap (Y_j \cup \OSA)\right).$$ Under the additional assumption that $Y \setminus (Y \cap \OSA)$ is a normal crossing divisor, the next theorem describes the cohomology at finite distance as the joint kernel of the residue morphism along the irreducible components of $Y$.

\begin{theorem}(\ref{theoremres})
\label{Residue_intro}
There is an isomorphism:
$$H^n(X \setminus \OSA, Y \setminus (Y \cap \OSA)) \simeq \ker \left( \bigoplus_{i \in I}\Res_{Y_i^\circ}: H^n(\BC^n \setminus \SA) \to \bigoplus_{i \in I}H^{n-1} \left( Y_i^\circ \right) \right).$$
\end{theorem}

Theorems \ref{char_intro} and \ref{Residue_intro} give two distinct descriptions of $H^n_{< \infty}(\BC^n \setminus \SA)$ as the kernel of a linear map defined on $H^n(\BC^n \setminus \SA).$ We show that these two descriptions coincide for a particular choice of compactifications of $\BC^n,$ called partial wonderful compactifications at infinity (see Definition \ref{lemcomp1} for a precise definition).
The classical definition of wonderful compactifications due to de Concini and Procesi in \cite{de1995wonderful} produces a compactification of $\BC^n \setminus \SA$ whose boundary is a normal crossing divisor. In our setting, however, we need only a compactification of $\BC^n$ satisfying the weaker condition of locally a product for $\SA$.
In particular, we only perform blow-ups along strata that lie at infinity. Accordingly, we adapt the construction of de Concini and Procesi as follows.
First, we compactify $\BC^n$ to $\BP^n_\BC$ by adding a hyperplane at infinity, denoted $L_\infty$. We obtain a projective arrangement $\tilde{\SA} \subseteq \BP^n_\BC.$ Then, we iteratively blow-up along a selected subset $\SG$ of strata of $\tilde{\SA} \cap L_\infty$ to obtain $X,$ the \emph{partial wonderful compactification at infinity}. It is locally a product for $\SA.$ The irreducible components of $Y=X \setminus \BC^n,$ the divisor at infinity, are indexed by $\SG$, i.e.,
$$Y = \bigcup_{S \in \SG}Y_S.$$
The set $\SG$ contains at least all zero-dimensional strata of $\tilde{\SA} \cap L_\infty;$ we denote this subset by $\SG_0.$
Equivalently, the set $\SG_0$ can be interpreted as the set of directions of the arrangement $\SA,$ a notion that depends only on the affine arrangement $\SA$.

\begin{figure}[h]
\tikzset{every picture/.style={line width=0.75pt}}

\begin{tikzpicture}[x=0.70pt,y=0.70pt,yscale=-1,xscale=1]

\draw    (303.82,67.04) -- (304.64,215.55) ;

\draw    (312.92,62) -- (258.42,201.94) ;

\draw    (220.45,113.98) -- (388.01,166.33) ;

\draw    (226.25,172.44) -- (391.34,158.06) ;

\draw [color={rgb, 255:red, 208; green, 2; blue, 27 }  ,draw opacity=1 ]   (293.91,73.35) -- (391.24,186.56) ;

\draw    (51,41.67) -- (51,226.67) ;
 
\draw    (121,41.67) -- (121,226.67) ;

\draw    (166,181) -- (5,180) ;

\draw    (169,111) -- (8,110) ;

\draw    (539,81.19) -- (539.64,223.55) ;

\draw    (528.91,81.35) -- (493.42,209.94) ;

\draw    (455.45,121.98) -- (630,158.19) ;

\draw    (461.25,180.44) -- (631,172.19) ;
 
\draw [color={rgb, 255:red, 208; green, 2; blue, 27 }  ,draw opacity=1 ]   (551,79.19) -- (633,144.19) ;

\draw  [draw opacity=0] (551.98,52.39) .. controls (554.68,54.47) and (557.06,57.05) .. (558.99,60.1) .. controls (567.86,74.09) and (563.7,92.63) .. (549.7,101.49) .. controls (535.7,110.36) and (517.17,106.2) .. (508.31,92.2) .. controls (506.38,89.15) and (505.06,85.89) .. (504.34,82.57) -- (533.65,76.15) -- cycle ; \draw  [color={rgb, 255:red, 245; green, 166; blue, 35 }  ,draw opacity=1 ] (551.98,52.39) .. controls (554.68,54.47) and (557.06,57.05) .. (558.99,60.1) .. controls (567.86,74.09) and (563.7,92.63) .. (549.7,101.49) .. controls (535.7,110.36) and (517.17,106.2) .. (508.31,92.2) .. controls (506.38,89.15) and (505.06,85.89) .. (504.34,82.57) ;  

\draw  [draw opacity=0] (627.48,188.09) .. controls (624.08,187.8) and (620.68,186.93) .. (617.41,185.42) .. controls (602.37,178.46) and (595.82,160.63) .. (602.77,145.6) .. controls (609.73,130.56) and (627.56,124) .. (642.59,130.96) .. controls (645.87,132.47) and (648.74,134.5) .. (651.15,136.9) -- (630,158.19) -- cycle ; \draw  [color={rgb, 255:red, 245; green, 166; blue, 35 }  ,draw opacity=1 ] (627.48,188.09) .. controls (624.08,187.8) and (620.68,186.93) .. (617.41,185.42) .. controls (602.37,178.46) and (595.82,160.63) .. (602.77,145.6) .. controls (609.73,130.56) and (627.56,124) .. (642.59,130.96) .. controls (645.87,132.47) and (648.74,134.5) .. (651.15,136.9) ;  

\draw    (16.68,214.76) -- (177.68,54.76) ;

\draw    (246.67,189.67) -- (374.67,78.67) ;

\draw    (478.52,202.98) -- (567.87,109.81) ;

\draw  [draw opacity=0] (605.2,103.02) .. controls (605.76,108.24) and (604.96,113.67) .. (602.59,118.78) .. controls (595.64,133.82) and (577.81,140.37) .. (562.77,133.42) .. controls (554.44,129.56) and (548.71,122.37) .. (546.44,114.17) -- (575.36,106.19) -- cycle ; \draw  [color={rgb, 255:red, 245; green, 166; blue, 35 }  ,draw opacity=1 ] (605.2,103.02) .. controls (605.76,108.24) and (604.96,113.67) .. (602.59,118.78) .. controls (595.64,133.82) and (577.81,140.37) .. (562.77,133.42) .. controls (554.44,129.56) and (548.71,122.37) .. (546.44,114.17) ;  

\draw  [color={rgb, 255:red, 245; green, 166; blue, 35 }  ,draw opacity=1 ][fill={rgb, 255:red, 245; green, 166; blue, 35 }  ,fill opacity=1 ] (300.33,84.97) .. controls (300.33,83.04) and (301.9,81.47) .. (303.83,81.47) .. controls (305.77,81.47) and (307.33,83.04) .. (307.33,84.97) .. controls (307.33,86.91) and (305.77,88.47) .. (303.83,88.47) .. controls (301.9,88.47) and (300.33,86.91) .. (300.33,84.97) -- cycle ;
 
\draw  [color={rgb, 255:red, 245; green, 166; blue, 35 }  ,draw opacity=1 ][fill={rgb, 255:red, 245; green, 166; blue, 35 }  ,fill opacity=1 ] (331.69,113.58) .. controls (333.56,114.05) and (334.7,115.95) .. (334.23,117.83) .. controls (333.75,119.7) and (331.85,120.84) .. (329.98,120.37) .. controls (328.1,119.89) and (326.97,117.99) .. (327.44,116.12) .. controls (327.91,114.24) and (329.82,113.11) .. (331.69,113.58) -- cycle ;

\draw  [color={rgb, 255:red, 245; green, 166; blue, 35 }  ,draw opacity=1 ][fill={rgb, 255:red, 245; green, 166; blue, 35 }  ,fill opacity=1 ] (365.33,159.97) .. controls (365.33,158.04) and (366.9,156.47) .. (368.83,156.47) .. controls (370.77,156.47) and (372.33,158.04) .. (372.33,159.97) .. controls (372.33,161.91) and (370.77,163.47) .. (368.83,163.47) .. controls (366.9,163.47) and (365.33,161.91) .. (365.33,159.97) -- cycle ;

\draw (190.82,93.55) node [anchor=north west][inner sep=0.75pt]   [align=left] {$\displaystyle \tilde{L}_{4}$};

\draw (377.53,197.42) node [anchor=north west][inner sep=0.75pt]  [color={rgb, 255:red, 208; green, 2; blue, 27 }  ,opacity=1 ] [align=left] {$\displaystyle L_{\infty }$};

\draw (38,241) node [anchor=north west][inner sep=0.75pt]   [align=left] {$\displaystyle L_{1}$};

\draw (111,243.17) node [anchor=north west][inner sep=0.75pt]   [align=left] {$\displaystyle L_{2}$};

\draw (1,185) node [anchor=north west][inner sep=0.75pt]   [align=left] {$\displaystyle L_{3}$};

\draw (11,120) node [anchor=north west][inner sep=0.75pt]   [align=left] {$\displaystyle L_{4}$};

\draw (570.53,72.42) node [anchor=north west][inner sep=0.75pt]  [color={rgb, 255:red, 208; green, 2; blue, 27 }  ,opacity=1 ] [align=left] {$\displaystyle Y_{L_\infty }$};

\draw (540,26) node [anchor=north west][inner sep=0.75pt]  [color={rgb, 255:red, 245; green, 166; blue, 35 }  ,opacity=1 ] [align=left] {$\displaystyle Y_{P_{12}}$};

\draw (617,193.17) node [anchor=north west][inner sep=0.75pt]  [color={rgb, 255:red, 245; green, 166; blue, 35 }  ,opacity=1 ] [align=left] {$\displaystyle Y_{P_{34}}$};

\draw (3,220) node [anchor=north west][inner sep=0.75pt]   [align=left] {$\displaystyle L_{5}$};

\draw (199.82,166.55) node [anchor=north west][inner sep=0.75pt]   [align=left] {$\displaystyle \tilde{L}_{3}$};

\draw (241.82,209.55) node [anchor=north west][inner sep=0.75pt]   [align=left] {$\displaystyle \tilde{L}_{1}$};

\draw (296.82,218.55) node [anchor=north west][inner sep=0.75pt]   [align=left] {$\displaystyle \tilde{L}_{2}$};

\draw (424.82,92.55) node [anchor=north west][inner sep=0.75pt]   [align=left] {$\displaystyle \overline{L}_{4}$};

\draw (430.82,169.55) node [anchor=north west][inner sep=0.75pt]   [align=left] {$\displaystyle \overline{L}_{3}$};

\draw (477.82,213.55) node [anchor=north west][inner sep=0.75pt]   [align=left] {$\displaystyle \overline{L}_{1}$};

\draw (533.82,226.55) node [anchor=north west][inner sep=0.75pt]   [align=left] {$\displaystyle \overline{L}_{2}$};

\draw (609.53,93.42) node [anchor=north west][inner sep=0.75pt]  [color={rgb, 255:red, 245; green, 166; blue, 35 }  ,opacity=1 ] [align=left] {$\displaystyle Y_{P_{5}}$};

\draw (453.82,196.55) node [anchor=north west][inner sep=0.75pt]   [align=left] {$\displaystyle \overline{L}_{5}$};

\draw (221.82,191.55) node [anchor=north west][inner sep=0.75pt]   [align=left] {$\displaystyle \tilde{L}_{5}$};

\draw (313,72) node [anchor=north west][inner sep=0.75pt]  [color={rgb, 255:red, 245; green, 166; blue, 35 }  ,opacity=1 ] [align=left] {$\displaystyle P_{12}$};

\draw (345,108) node [anchor=north west][inner sep=0.75pt]  [color={rgb, 255:red, 245; green, 166; blue, 35 }  ,opacity=1 ] [align=left] {$\displaystyle P_{5}$};

\draw (370,131) node [anchor=north west][inner sep=0.75pt]  [color={rgb, 255:red, 245; green, 166; blue, 35 }  ,opacity=1 ] [align=left] {$\displaystyle P_{34}$};

\end{tikzpicture}

\caption{An example of a partial wonderful compactification at infinity. From left to right, the arrangements $\SA$ in $\BC^2$, $\tilde{\SA}$ in $\BP_\BC^2$ and $\OSA$ in $X$.}
\label{figrun1}
\end{figure}
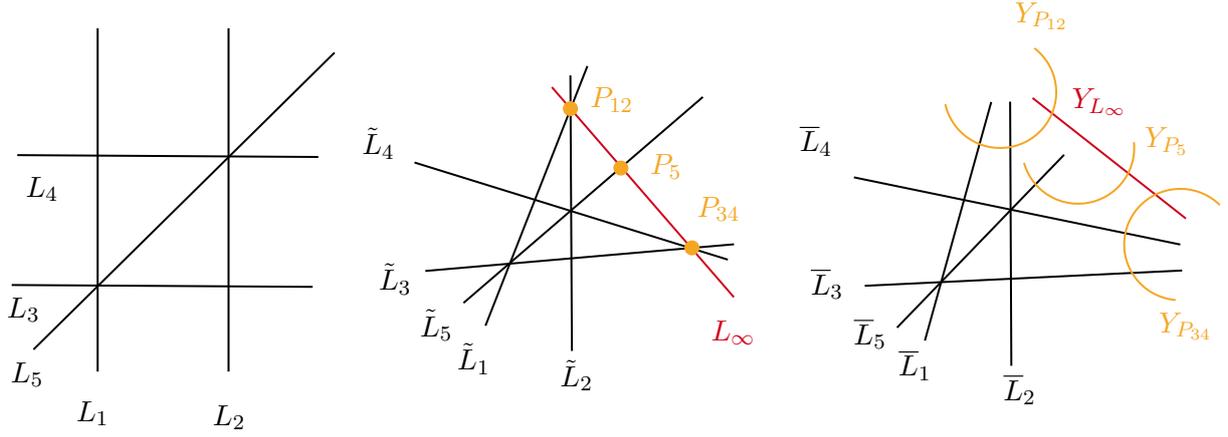

The following theorem shows that, for partial wonderful compactifications at infinity, Theorems \ref{char_intro} and \ref{Residue_intro} provide equivalent descriptions. Moreover, it can be viewed as a refinement of Theorem \ref{Residue_intro} in the sense that it suffices to consider the residues on the components of $Y$ indexed by $\SG_0$ and rather than on all components as in Theorem \ref{Residue_intro}.

\begin{theorem}
\ref{lem1}
\label{lem1_intro}
The following diagram commutes:
\[\begin{tikzcd}[row sep=huge,column sep=huge]
	{H^n(\BC^n \setminus \SA)} & &{H^{n-1}(\BC^n \setminus \SA)} \\
	{H^n(X \setminus \OSA, Y \setminus (Y \cup \OSA))} & &{\displaystyle{\bigoplus_{P \in \SG_0}H^{n-1}(Y^\circ_P).}}
	\arrow["\partial", from=1-1, to=1-3]
	\arrow["\simeq", from=1-1, to=2-1]
	\arrow["\simeq", from=1-3, to=2-3]
	\arrow["{\displaystyle{\bigoplus_{P \in \SG_0}\Res_{Y_P^\circ}}}"', from=2-1, to=2-3]
\end{tikzcd}\]
\end{theorem}

The upper row of this diagram corresponds the map of Theorem \ref{char_intro}, while the lower row is a refinement of the map appearing in Theorem \ref{Residue_intro}. The theorem then shows that the two descriptions of $H^n_{< \infty}(\BC^n \setminus \SA)$ in terms of a kernel of a linear map with source $H^n(\BC^n \setminus \SA)$ given in Theorems \ref{char_intro} and \ref{Residue_intro} are equivalent.

To illustrate the construction of a partial wonderful compactification at infinity, we consider the running Example \ref{runningexample}.
Further details will be provided in Example \ref{example1}.

\begin{example}
Figure \ref{figrun1} is an illustration of the situation. First, we compactify $\BC^2$ to $\BP^2_\BC$ by adding a hyperplane $L_\infty$ at infinity. We obtain an arrangement $\tilde{\SA}=\{\tilde{L}_1, \tilde{L}_2, \tilde{L}_3, \tilde{L}_4, \tilde{L}_5\} \subseteq \BP^n_\BC.$ The hyperplanes of $\tilde{\SA}$ intersect $L_\infty$ in three points: $$P_{12}=\tilde{L}_1 \cap \tilde{L}_2 \cap L_\infty, \ P_{34}=\tilde{L}_3 \cap \tilde{L}_4 \cap L_\infty, \text{ and } P_{5}=\tilde{L}_5 \cap L_\infty.$$
To construct the wonderful compactification at infinity $X,$ we blow-up these three points and denote by $Y_{P_{12}},Y_{P_{34}}$ and $Y_{P_5}$ the exceptional divisors of the blow-ups at the points $P_{12},$ $P_{34}$ and $P_{5},$ respectively. We observe that $$H^1(\BC^2 \setminus \SA) \simeq H^1(Y^\circ_{P_{12}}) \oplus H^1(Y_{P_{34}}^\circ) \oplus H^1(Y_{P_{5}}^\circ),$$ 
which can be viewed as a variant of Brieskorn decomposition. In Example \ref{example1}, we present computations that verify that for $\SG_0:=\{P_{12}, P_{34}, P_5\},$ we have
$$\ker(\partial) \simeq \ker\left(\bigoplus_{P \in \SG_0}\Res_{Y_P^\circ}\right).$$
Note that, to obtain a wonderful compactification in the usual sense, one would need to further perform blow-ups at the points $P_{135}=\overline{L}_1 \cap \overline{L}_3 \cap \overline{L}_5$ and $P_{245}=\overline{L}_2 \cap \overline{L}_4 \cap \overline{L}_5$.
\end{example}

Lastly, we turn to the notion of canonical forms, which Arkani-Hamed, Bai, and Lam introduced in \cite{PosGeom}, motivated by particle physics and cosmology. In \cite{brown2025positive}, the authors use mixed Hodge theory to define a canonical map for pairs of varieties $(X,Y)$ where $X$ is compact and of dimension $n$, which assigns to each class in the homology group $H_n(X,Y)$, a class of differential forms in the cohomology group $H^n(X \setminus Y)$. In the case of a projective hyperplane arrangement $\SP \subseteq \BP_\BC^n$, this canonical map is given by Poincaré duality. In Section \ref{Section4}, we extend this definition to the case of hyperplane arrangements in the non-compact ambient variety $\BC^n$ using Poincaré duality \eqref{Poincare1}, and we explain how our construction is related to the compact case.

This perspective was our initial motivation for studying the cohomology at finite distance. Indeed, when $\SA$ is an essential affine hyperplane arrangement, the relevant canonical form naturally takes value in the cohomology group $H^n_{< \infty}(\BC^n \setminus \SA).$ An interesting class of examples is the integrand of the cosmological correlators as defined in \cite{cosmological}.

\subsection{Structure of the paper}
This paper is organized as follows.
In Section \ref{section2.0}, we review the background material on Orlik--Solomon algebras.

In Section \ref{section2.1}, we prove Theorem \ref{char_intro}. In Section \ref{section2.3}, we show that for the complexification of real arrangements the bounded regions form a basis of the relative homology.

After expressing the cohomology at finite distance in terms of a compactification $X$ having suitable properties at infinity in Section \ref{section3.1}, we prove Theorem \ref{Residue_intro} in Section \ref{section3.11}. Section \ref{section2.2} constructs the partial wonderful compactification at infinity.
In Section \ref{section2.22}, we prove Theorem \ref{lem1_intro}.
Section \ref{Section3.4} is devoted to examples.

In Section \ref{sec4.1}, we introduce a definition for the canonical map. Lastly, in Section \ref{sec4.2}, we derive a formula for the canonical form associated with a bounded region (see Proposition \ref{prop6.7}) for complexifications of real arrangements.

\subsection{Conventions}
We denote by $A_\bullet(\SA)$ the Orlik--Solomon algebra of an affine hyperplane arrangement $\SA$ with rational coefficients. Similarly, all cohomology and homology groups will be taken with rational coefficients.

\subsection{Acknowledgements}
The author thanks Clément Dupont for suggesting the topic of this article and for many helpful ideas and discussions, and Olof Bergvall, Emanuele Delucchi, Claudia Fevola, Cary Malkiewich, Marta Panizzut, Bernd Sturmfels, and Claudia Yun for insightful discussions.

\section{Preliminaries on Orlik--Solomon Algebras}
\label{section2.0}

In this section, we recall the standard definitions and results about Orlik--Solomon algebras that will be used throughout the paper. These concepts were initially introduced and developed by Orlik and Solomon in \cite{Orlik-Solomon} and by Looijenga in \cite{Looijenga}. For additional explanations, we also refer to the textbooks \cite{Hyperplane} and \cite{Hypstanley}.

\subsection{The case of a central arrangement}

We start with central hyperplane arrangements $\SA \subseteq \BC^n$, i.e., arrangements in which all hyperplanes share a common intersection.

\begin{definition}
\label{Orlik--solomondef}
Let $\SA=\{L_1, \ldots, L_m\} \subseteq \BC^n$ be a central hyperplane arrangement. For each hyperplane $L_i$, where $1 \leq i \leq m$, we associate a generator $e_i$.
The Orlik--Solomon algebra with rational coefficients is defined as:
$$A_\bullet(\SA)=\bigwedge \nolimits^{\bullet} (e_{1}, \ldots, e_{m})/I_\bullet(\SA),$$
where $I_\bullet(\SA)$ is the ideal of the exterior algebra generated by the relations:
$$ \sum_{j=1}^{k+1} (-1)^je_{i_1} \wedge \cdots \wedge \widehat{e}_{i_j} \wedge \cdots \wedge e_{i_{k+1}}$$ 
for $L_{i_1},\ldots, L_{i_{k+1}}$ linearly dependent hyperplanes.
This is a graded algebra inheriting its grading from that of the exterior algebra.
\end{definition}

The Orlik--Solomon algebra provides a combinatorial description of the cohomology of the complement of the hyperplane arrangement $\SA$.

\begin{prop}\cite[Theorem 5.2]{Orlik-Solomon}
\label{Orlik--Solomon2}
The cohomology of the complement is isomorphic to the Orlik--Solomon algebra: $$A_\bullet(\SA) \simeq H^\bullet(\BC^n \setminus \SA).$$
Moreover, for $L_i=\{f_i=0\} \in \SA$, this isomorphism sends $e_i$ to the class of the differential form $d\log(f_i).$
\end{prop}

\begin{lemma}
\label{lemmacomplex}
The homogeneous components $A_k(\SA)$ form a complex
$$0 \to A_n(\SA) \overset{\partial}{\to} A_{n-1}(\SA) \overset{\partial}{\to} \cdots \overset{\partial}{\to} A_1(\SA) \overset{\partial}{\to} A_0(\SA) \to 0$$
where the boundary operator $\partial$ is defined as follows:
\begin{align*}
\partial: A_k(\SA) &\to A_{k-1}(\SA) \\
e_{i_1} \wedge \cdots \wedge e_{i_k} & \mapsto \sum_{j=1}^{k} (-1)^{j+1}e_{i_1} \wedge \cdots \wedge \widehat{e}_{i_j} \wedge \cdots \wedge e_{i_k}.
\end{align*}
\end{lemma}

\begin{proof}
A direct computation shows that $\partial$ is well-defined and $\partial \circ \partial=0.$
\end{proof}

\begin{remark}
This map $\partial$ is derivation, i.e., for $a \in A_j(\SA)$ and $b \in A_k(\SA),$ we have $\partial(a \wedge b)=\partial(a) \wedge b+(-1)^j a \wedge \partial(b).$ Moreover, it is the unique derivation satisfying $\partial(e_i)=1$. 
\end{remark}

\begin{lemma}
\label{lemmasequence}
If $\SA$ is non-empty, the complex
$$0 \to A_n(\SA) \xrightarrow{\partial} A_{n-1}(\SA) \xrightarrow{\partial} \cdots \xrightarrow{\partial} A_1(\SA) \xrightarrow{\partial} A_0(\SA) \xrightarrow{\partial} 0$$
is an exact sequence.
\end{lemma}

\begin{proof}
A contracting homotopy is given by $h(x)=e_1 \wedge x.$
\end{proof}

\subsection{The case of an affine arrangement}
We now consider affine hyperplane arrangements $\SA \subseteq \BC^n$ which admit analogous definitions and results.

\begin{definition}
Let $\SA=\{L_1, \ldots, L_m\} \subseteq \BC^n$ be an affine hyperplane arrangement. For each hyperplane $L_i$, where $1 \leq i \leq m$, we associate a generator $e_i$.
The Orlik--Solomon algebra is defined as follows:
$$A_\bullet(\SA)=\bigwedge \nolimits^{\bullet} (e_{1}, \ldots, e_{m})/I_\bullet(\SA),$$
where the ideal of the exterior algebra $I_\bullet(\SA)$ is given by
$$\begin{pmatrix}
    \sum_{j=1}^{k+1} (-1)^{j+1}e_{i_1} \wedge \cdots \wedge \widehat{e}_{i_j} \wedge \cdots \wedge e_{i_{k+1}} \ | \ L_{i_1},\ldots, L_{i_{k+1}} \text{ are linearly dependent} \\
    e_{i_1} \wedge \cdots \wedge e_{i_k} \ | \ L_{i_1} \cap \cdots \cap L_{i_{k}}= \emptyset
\end{pmatrix}$$
This is a graded algebra inheriting its grading from that of the exterior algebra.
\end{definition}

We also have a version of Proposition \ref{Orlik--Solomon2} for affine hyperplane arrangements.

\begin{prop}\cite[Theorem 5.2]{Orlik-Solomon}
\label{Orlik--Solomon}
The cohomology of the complement is isomorphic to the Orlik--Solomon algebra: $$A_\bullet(\SA) \simeq H^\bullet(\BC^n \setminus \SA).$$
Moreover, for $L_i=\{f_i=0\} \in \SA,$ this isomorphism sends $e_i$ to the class of the differential form $d\log(f_i).$
\end{prop}

An alternative description of the homogeneous components of the Orlik--Solomon algebra is provided by the Brieskorn decomposition. This decomposition expresses each homogeneous component of an affine hyperplane arrangement as a direct sum of homogeneous components corresponding to central subarrangements arising from intersections of hyperplanes in $\SA$. To formulate this decomposition, we first recall the definition of a stratum.

\begin{definition}
Let $\SA=\{L_1, \ldots, L_m\}$ be an affine hyperplane arrangement. A stratum of $\SA$ of codimension $p$ is the intersection of $k$ hyperplanes $S:=L_{i_1} \cap  \cdots \cap L_{i_k}$ such that $S$ has codimension $p$ in $\BC^n.$
\end{definition}

The strata of the arrangement $\SA$ form a poset $\SL(\SA)$ with order given by reverse inclusion. It has the least element $\BC^n$ and the maximal element the empty set $\emptyset$. We denote by $\SL_p(\SA)$ the set of all strata of codimension $p$.

The decomposition in the following proposition is known as the Brieskorn decomposition.

\begin{prop}\cite[Equation 2.11]{Orlik-Solomon}
For each stratum $S$ of codimension $p$, let $\SA_S$ denote the central subarrangement $\SA_S=\{L \in \SA, S \subseteq L\}$ of $\SA$. Then,
$$ A_p(\SA)= \bigoplus_{S \in \SL_p(\SA)} A_p(\SA_S).$$
\end{prop}

The following definition is an analogue of the boundary map for affine hyperplane arrangements.

\begin{lemma}
\label{defsequence}
The homogeneous components $A_k(\SA)$ form a complex:
$$0 \to A_n(\SA) \xrightarrow{\partial} A_{n-1}(\SA) \xrightarrow{\partial} \cdots \xrightarrow{\partial} A_1(\SA) \xrightarrow{\partial} A_0(\SA) \xrightarrow{\partial} 0,$$
where the boundary operator $\partial: A_k \to A_{k-1}$ is the direct sum of $\partial_S: A_k(\SA_S) \to A_{k-1}(\SA_T)$ for $S \subseteq T$ and $k=\codim(S)=\codim(T)+1.$ Here, $\partial_S$ is the derivation map defined in Lemma \ref{lemmacomplex} for the central subarrangement $\SA_S$.
\end{lemma}

\begin{proof}
The result follows from Lemma \ref{lemmacomplex}.
\end{proof}

\begin{remark}
	Note that if the arrangement is not central, this is not a derivation of the Orlik--Solomon algebra $A_\bullet(\SA)$.
\end{remark}

We conclude this section by recalling the definition of the characteristic polynomial. It encodes several invariants of the arrangement. For instance, when $\SA$ is the complexification of a real arrangement $\SA_\BR \subseteq \BR^n$, the characteristic polynomial determines the number of regions and bounded regions determined of the real arrangement (see Zaslavsky's theorem \cite[Theorem C]{zaslavsky1997facing}).

\begin{definition}
Let $\mu_{\SA}$ denote the Möbius function of the poset of strata $\SL(\SA)$. The characteristic polynomial of the arrangement $\chi_\SA(t)$ is defined by
$$\chi_\SA(t) = \sum_{S \in \SL(A)} \mu_\SA(S)t^{\dim(S)}.$$
\end{definition}

\subsection{The case of a projective arrangement}
This section concerns projective hyperplane arrangements $\SP=\{P_1, \ldots, P_m, P_{m+1}\} \subseteq \BP^n_\BC$. Their Orlik--Solomon algebra is constructed from the central case as follows: consider the affine cone over $\BP^n_\BC$, which yields a central arrangement $\tilde{\SP}=\{\tilde{P}_1, \ldots,\tilde{P}_m,\tilde{P}_{m+1} \}\subseteq \BC^{n+1}.$ The Orlik--Solomon algebra of the central arrangement $\tilde{\SP}$ is given by $$A_\bullet(\tilde{\SP})=\bigwedge \nolimits^{\bullet} (\tilde{e}_1, \ldots,\tilde{e}_m, \tilde{e}_{m+1})/ I_\bullet(\tilde{\SP}),$$
where $I_\bullet(\tilde{\SP})$ is as defined in Definition \ref{Orlik--solomondef}. Moreover, we have the derivation map $$\partial: A_\bullet(\tilde{\SP}) \to A_{\bullet-1}(\tilde{\SP})$$ introduced in Lemma \ref{lemmacomplex}.

\begin{definition}
The Orlik--Solomon algebra of a projective hyperplane arrangement $\SP \subseteq \BP^n_\BC$ is defined as $$A_\bullet(\SP) \simeq \ker\left(\partial: A_\bullet(\tilde{\SP}) \to A_{\bullet-1}(\tilde{\SP})\right).$$
\end{definition}

As before, there is an isomorphism between the Orlik--Solomon algebra and the cohomology of the complement.

\begin{prop}\cite[Theorem 5.2]{Orlik-Solomon}
The cohomology of the complement is isomorphic to the Orlik--Solomon algebra:
$$A_\bullet(\SP) \simeq H^\bullet(\BP^n_\BC \setminus \SP).$$
Moreover, for $P_i=\{f_i=0\} \in \SP,$ this isomorphism sends $\tilde{e}_i-\tilde{e}_j$ to the class of the differential form $d \log(f_i/f_j).$
\end{prop}

The Orlik--Solomon algebra of a projective hyperplane arrangement is constructed from the central hyperplane arrangement corresponding to its affine cone. It is also related to affine hyperplane arrangements as follows.
Consider $P_{m+1}$ as the hyperplane at infinity and for $1 \leq i \leq m$, set $L_i=\{f_i/f_{m+1}=0\} \subseteq \BC^n$. Therefore, we have $$\BP^n_\BC \setminus P_{m+1} \simeq \BC^n \text{ and } \BP^n_\BC \setminus \SP \simeq \BC^n \setminus \SA,$$ where $\SA=\{L_1, \ldots, L_m\}.$ In particular, there is an isomorphism: $$A_\bullet(\SA)\simeq A_\bullet(\SP),$$ which is given by sending $e_i$ to $\tilde{e_i}- \tilde{e}_{m+1}.$

\section{Cohomology at finite distance}
\label{section2}

\subsection{Cohomology at finite distance of a complex arrangement}
\label{section2.1}
In this section, we introduce the main object of study, the cohomology group $H^n_{< \infty}(\BC^n \setminus \SA).$ We first establish the isomorphism
$$H^k_{< \infty}(\BC^n \setminus \SA) \simeq H_{2n-k}(\BC^n, \SA)$$
arising from Poincaré duality, and then present an Orlik--Solomon-type description of $H^n_{< \infty}(\BC^n \setminus \SA).$
Lastly, we prove that the relative homology and its Poincaré dual, the cohomology at finite distance, vanish for all $k \ne n$, and derive a formula for the dimension $\dim(H_n(\BC^n, \SA))=\dim(H^n_{< \infty}(\BC^n \setminus \SA)).$

\begin{remark}
From this point forward, we will only consider essential affine hyperplane arrangements, that is, arrangements with a zero-dimensional stratum. However, if $\SA \subseteq \BC^n$ is not essential and $r$ denotes the minimal dimension of a stratum, one can obtain an essential arrangement in $\BC^r$ by intersecting with an $(n-r)$-dimensional affine subspace of $\BC^n,$ that is generic with respect to $\SA.$ In particular, all results stated for essential arrangements also apply to non-essential ones, up to a shift in the dimension.
\end{remark}

\begin{definition}
Let $\SA \subseteq \BC^n$ be an essential affine hyperplane arrangement and let $j: \BC^n \setminus \SA \to \BC^n$ denote the open immersion. We define the cohomology at finite distance with rational coefficients as $$H^n_{< \infty}(\BC^n \setminus \SA):=H^n_c(\BC^n, Rj_*\BQ_{\BC^n \setminus \SA}).$$
\end{definition}

\begin{lemma}
We have an isomorphism between the cohomology at finite distance of the complement and the relative cohomology:
$$H^n_{< \infty}(\BC^n \setminus \SA) \simeq H_n(\BC^n, \SA).$$
\end{lemma}

\begin{proof}
Let $a: \BC^n \to \{\star\}$ denote the map to a point. Using the language of derived functors, the cohomology at finite distance can be expressed as $$H^n_{< \infty}(\BC^n \setminus \SA)=H^n_c(\BC^n, Rj_*\BQ_{\BC^n \setminus \SA})=H^n( Ra_!Rj_*\BQ_{\BC^n \setminus \SA})$$ and the relative homology as $$H_n(\BC^n, \SA)=(H^n(\BC^, \SA))^\vee =H^n\left((Ra_*Rj_! \BQ_{\BC^n \setminus \SA})\right)^\vee.$$ The result then follows by Verdier duality.
\end{proof}

\begin{theorem}
\label{thmcharacterization}
Let $\SA \subseteq \BC^n$ be an essential affine hyperplane arrangement.
Then,
$$H^n(\BC^n, \SA)\simeq \coker\left(\partial^\vee: A_{n-1}^\vee \to A_n^\vee\right).$$
Dually,
$$H_n(\BC^n, \SA) \simeq H^n_{< \infty}(\BC^n \setminus \SA) \simeq \ker\left(\partial: A_n(\SA) \to A_{n-1}(\SA)\right).$$
\end{theorem}

\begin{proof}
We will make use of the following exact sequence of sheaves, which was first introduced by Looijenga in \cite[Section 2]{Looijenga}
(note that the $A_k(\SA_S)$ are denoted $E_S$ in his notation):
$$0 \longrightarrow j_!(\SF_{|\BC^n \setminus \SA}) \longrightarrow \SF \longrightarrow \bigoplus_{S \in \SL_1(\SA)} (i_S)_*(\SF_{|S}) \otimes A_1(\SA_S)^\vee \longrightarrow \bigoplus_{S \in \SL_2(\SA)} (i_S)_*(\SF_{|S}) \otimes A_2(\SA_S)^\vee \longrightarrow \cdots$$
where $i_S: S \to \BC^n$ is the closed immersion corresponding to a stratum $S$ and $\SF$ is a sheaf on $\BC^n.$
This sequence induces a spectral sequence (see \cite[Equation (2.3-1)]{Looijenga}):
$$E_1^{p,q}=\bigoplus_{S \in \SL_p(\SA)} H^q(S, \SF_{|S}) \otimes A_p(\SA_S)^\vee \implies H^{p+q}(X, \SA, \SF).$$
Specializing to $X=\BC^n$ and $\SF= \BQ_{\BC^n}$, the constant sheaf on $\BC^n$, the $E_1$-page vanishes for $q \ne 0$, i.e.,: 
$$E_1^{p,0}=\bigoplus_{S \in \SL_p(\SA)} A_p(\SA_S)^\vee= A_p(\SA)^\vee \implies H^{p}(\BC^n, \SA).$$
In particular, the relative cohomology can be computed as
$$H^p(\BC^n, \SA)\simeq H^p(A_\bullet^\vee, \partial^\vee),$$
which is the cohomology of the complex:
$$0 {\longrightarrow} A_0(\SA)^\vee \overset{\partial^\vee}{\longrightarrow} A_1(\SA)^\vee \overset{\partial^\vee}{\longrightarrow} \cdots \overset{\partial^\vee}{\longrightarrow} A_n(\SA)^\vee \overset{\partial^\vee}{\longrightarrow} 0.$$
Dually, the cohomology at finite distance is given by
$$H_p(\BC^n, \SA) \simeq H^{2n-p}_{< \infty}(\BC^n, \SA)\simeq H^p(A_\bullet, \partial),$$
which is the homology of the complex:
$$0 \longrightarrow A_n(\SA) \overset{\partial}{\longrightarrow} A_{n-1}(\SA) \overset{\partial}{\longrightarrow} \cdots \overset{\partial}{\longrightarrow} A_0(\SA) \overset{\partial}{\longrightarrow} 0.$$
In middle degree $p=n,$ this gives $H^n(\BC^n, \SA)\simeq \coker(\partial_{n}^\vee)$ and dually,
\begin{equation*}
H_n(\BC^n, \SA)\simeq H^n_{< \infty}(\BC^n \setminus \SA) \simeq \ker(\partial_n). \qedhere
\end{equation*}
\end{proof}

\begin{remark}
\label{remark}
This theorem in particular implies that $H^n_{< \infty}(\BC^n \setminus \SA)$ embeds into $H^n(\BC^n, \SA)$ since $\ker(\partial) \subseteq A_n(\SA).$
\end{remark}

\begin{prop}
\label{propchar}
For an essential affine hyperplane arrangement $\SA$, we have $$H^{2n-k}_{\infty}(\BC^n \setminus \SA) \simeq H_k(\BC^n, \SA) = 0 \text{ for } k \ne n,$$ and the dimension of the unique non-zero groups is given by:
$$\dim(H^n_{< \infty}(\BC^n \setminus \SA))=\dim(H_n(\BC^n, \SA))=(-1)^n\chi_{\SA}(1).$$
\end{prop}

\begin{remark}
Note that the dimension of the cohomology of the complement is also encoded in the characteristic polynomial (see \cite[Theorem 3.68]{Orlik-Terao}). Indeed, $$\dim(H^n(\BC^n, \SA))=\chi_\SA(0).$$
\end{remark}

\begin{proof} 
We will proceed by induction on the number $m$ of hyperplanes in $\SA.$ The base case is the hyperplane arrangement $\SA_0=\{L_1,\ldots, L_n\} \subseteq \BC^n$ where the intersection $L_1 \cap \cdots \cap L_n$ is a point. In this situation, the arrangement is central, which implies by Lemma \ref{lemmasequence} that
$$H_k(A_\bullet(\SA_0), \partial)=0 \text{ for } k<n.$$

Let us now assume that for every $j \in \BN$ and every essential hyperplane arrangement $\SA \subseteq \BC^j$ with $m'<m$ hyperplanes, we have $$H_{k}(\BC^j, \SA)=H_{k}(A_\bullet(\SA), \partial)=0 \text{ for all } k<j.$$ Let $\SA \subseteq \BC^n$ be an essential hyperplane arrangement with $m>n$ hyperplanes. Let us choose a hyperplane $L$
such that $\SA':=\SA \setminus \{L\}$ and $\SA'':=\SA \cap L,$ the latter viewed as a hyperplane arrangement on $L$, are both essential. Such a choice of $L$ is always possible provided $m>n.$
One obtains the restriction-deletion short exact sequence (see \cite[Theorem 3.65]{Orlik-Terao}):
$$0 \to A_k(\SA') \to A_k(\SA) \to A_{k-1}(\SA'') \to 0.$$
It induces an exact sequence of complexes:
$$0 \to A_\bullet(\SA') \to A_\bullet(\SA) \to A_{\bullet -1}(\SA'') \to 0,$$
where the complexes are those from Definition \ref{defsequence}.
One can verify that every square of this diagram commutes. It induces the following long exact sequence in homology:
\begin{align}
\label{seq1}
    \cdots & \to H_{k}(A_\bullet(\SA''), \partial) \to  H_k(A_\bullet(\SA'), \partial)  \to H_k(A_\bullet(\SA), \partial)  \to \\
	&\to H_{k-1}(A_\bullet(\SA''), \partial) \to H_{k-1}(A_\bullet(\SA'), \partial) \to  \cdots
\end{align}
Both $\SA'$ and $\SA''$ contain strictly fewer than $m$ hyperplanes. By the induction hypothesis, we have
$$H_k(A_\bullet(\SA'), \partial)=0 \text{ for all } k<n \text{ and }H_k(A_\bullet(\SA''), \partial)=0 \text{ for all } k<n-1.$$
Using the exact sequence \eqref{seq1},
we conclude that \begin{align}
\label{eq1}
H^{2n-k}_{< \infty}(\BC^n \setminus \SA) \simeq H_k(\BC^n, \SA) \simeq H_k(A_\bullet(\SA), \partial)=0 \text{ for all } k<n.
\end{align}

Lastly, we prove that $$\dim(H^n_{< \infty}(\BC^n \setminus \SA)) = \dim(H_n(\BC^n, \SA))=(-1)^n\chi_\SA(1).$$
From \cite[Theorem 2.5]{Hypstanley}, the characteristic polynomial satisfies the following recurrence relation: for the essential hyperplane arrangements $\SA'$ and $\SA''$, as defined before, we have
$$(-1)^n\chi_{\SA}(t)=(-1)^n\chi_{\SA'}(t)+(-1)^{n-1}\chi_{\SA''}(t).$$
On the other hand, the vanishing of the cohomology groups \eqref{eq1} reduces the sequence \eqref{seq1} to a short exact sequence, which gives
$$\dim(H_n(\BC^n, \SA))=\dim(H_n(\BC^n, \SA'))+ \dim(H_{n-1}(L, \SA'')).$$
By induction on the number of hyperplanes, it remains to verify the base case: $\SA_0 \subseteq \BC^n$ consists of exactly $n$ hyperplanes meeting in a point. In this case, it is clear that
$$\begin{cases}
	\dim(H_n(\BC^n, \SA_0))=0=(-1)^n\chi_{\SA_0}(1) \text{ if } n>0; \\
	\dim(H_n(\BC^n, \SA_0))=1=(-1)^n\chi_{\SA_0}(1) \text{ if } n=0.
\end{cases}$$
This completes the proof.
\end{proof}

\subsection{Relative homology of a real arrangement}
\label{section2.3}

We now turn to essential affine hyperplane arrangements $\SA$ obtained by complexifying a real arrangement $\SA_\BR \subseteq \BR^n.$ Zaslavsky's theorem \cite[Theorem C]{zaslavsky1997facing} states that the number of bounded regions of the real arrangement $\SA_\BR$ is equal to $(-1)^n\chi_{\SA}(1),$ which with the dimension of $H_n(\BC^n, \SA)$ by Proposition \ref{propchar}. This is not a coincidence: we will show that the bounded regions form a basis of the relative homology. Note that the proof relies on Zaslavsky's theorem.

\begin{prop}
\label{propboundedreg}
If the arrangement $\SA$ is the complexification of a real arrangement $\SA_\BR$ in $\BR^n$, the bounded regions of the real arrangement $\SA_{\BR}$ form a basis of the homology group $H_n(\BC^n, \SA).$
\end{prop}

\begin{remark}
Proposition \ref{propboundedreg} can also be derived from \cite[Theorem B]{kupers}, which uses the language of Tits buildings.
\end{remark}

\begin{proof}
We follow the approach outlined in \cite[Proposition 6.2]{brown2025positive}.
We proceed by induction on the number $m$ of hyperplanes in $\SA$. The base case occurs when $\SA_0 \subseteq \BC^n$ is an arrangement consisting of exactly $n$ hyperplanes meeting in a point. In this situation, the arrangement is central. Hence, if $n>0,$ then $H_n(\BC^n, \SA_0)=0$ by Lemma \ref{lemmasequence} and Theorem \ref{theoremres} and the real arrangement $(\SA_0)_{\BR}$ has no bounded regions. If $n=0,$ then $H_n(\BC^n, \SA_0)$ is generated by a point which is the unique bounded region of the real arrangement.

Let us now assume that for every $j \in \BN$ and every essential hyperplane arrangement $\SA \subseteq \BC^j$ with $m'<m$ hyperplanes, the bounded regions of $\SA_{\BR}$ form a basis of the homology group $H_j(\BC^j, \SA).$ Consider an essential hyperplane arrangement $\SA \subseteq \BC^n$ with $m$ hyperplanes. As in the proof of Proposition \ref{propchar}, we choose a hyperplane $L \in \SA$ such that $\SA':= \SA \setminus \{L\} \subseteq \BC^n$ and $\SA'':= \SA \cap L \subseteq L$ are both essential. As noted earlier, if $\SA$ consists of $m>n$ hyperplanes, such a choice of $L$ is always possible. Both $\SA'$ and $\SA''$ contain strictly fewer than $m$ hyperplanes. Let $\SB, \SB'$ and $\SB''$ denote the set of bounded regions of $\SA_\BR, \SA'_\BR$ and $\SA''_\BR,$ respectively. We then consider the following diagram:
\[\begin{tikzcd}
	0 & {H_n(\BC^n, \SA')} & {H_n(\BC^n, \SA)} & {H_{n-1}(L, \SA'')} & 0 \\
	0 & {\BQ \SB'} & {\BQ \SB} & {\BQ \SB''} & 0
	\arrow[from=1-1, to=1-2]
	\arrow[from=1-2, to=1-3]
	\arrow[from=1-3, to=1-4]
	\arrow[from=1-4, to=1-5]
	\arrow[from=2-1, to=2-2]
	\arrow[from=2-2, to=1-2]
	\arrow["f",from=2-2, to=2-3]
	\arrow[from=2-3, to=1-3]
	\arrow["g",from=2-3, to=2-4]
	\arrow[from=2-4, to=1-4]
	\arrow[from=2-4, to=2-5]
\end{tikzcd}\]

The first row is the long exact sequence in relative cohomology; by Proposition \ref{propchar},
it reduces to a short exact sequence.
The vertical maps send a bounded region to its class in homology. We define the maps $f$ and $g$ exactly as in \cite[Propostion 6.2]{brown2025positive}. As in loc.~cit., these definitions make both square commutes, and one verifies that $f$ is injective and $\im(f)= \ker(g )$.

Since the leftmost and rightmost vertical maps are isomorphisms, it follows that the middle vertical map $\BQ \SB \to H_n (\BC^n, \SA)$ is injective. By Zaslavsky's theorem (\cite[Theorem C]{zaslavsky1997facing}), the number of bounded regions is given by $(-1)^n\chi_\SA(1)$ which coincides, by Proposition \ref{propchar}, with the dimension $\dim(H_n(\BC^n, \SA))$. A dimension count shows that the bounded regions of $\SA_\BR$ form a basis of $H_n(\BC^n, \SA).$
\end{proof}

\section{Residues at infinity}
\label{Section3}

\subsection{Cohomology at finite distance in terms of a compactification}
\label{section3.1}

In this section, we express the cohomology at finite distance $H^n_{< \infty}(\BC^n \setminus \SA)$ in terms of a compactification. This reformulation is necessary to make the notion of residue at infinity precise.

\begin{definition}
\label{locallyaproduct}
Let $X$ be a smooth compactification of $\BC^n$, $\OSA$ be the closure of $\SA$ in $X$ and $Y:=X \setminus \BC^n$ be the divisor at infinity. We say that $X$ is locally a product for $\SA$ if for each point $P$ of $\OSA \cap Y,$ there exist an open $U \subseteq X$ around $P,$ a decomposition $U=Z_1 \times Z_2,$ and two closed $A \subseteq Z_1$, and $B \subseteq Z_2$ such that $\OSA \cap U=A \times Z_2$ and $Y \cap U=Z_1 \times B.$
\end{definition}

\begin{remark}
Note that in the above definition, $X$ is a complex algebraic variety and the notions of openness and closedness refer to the analytic topology.
\end{remark}

The proof of the following lemma relies on the assumption that $X$ is locally a product for $\SA.$ This condition is necessary in order to apply \cite[Proposition 3.2]{Single-valued} and is weaker than the normal crossing condition at infinity.

\begin{lemma}
\label{lemmainclusion}
Let $X$ be a compactification of $\BC^n$ that is locally a product for $\SA,$ and let $Y$ and $\OSA$ be as defined above.
Then, the following diagram commutes:
\begin{equation}
\label{eq2}
\begin{tikzcd}[row sep=huge,column sep=large]
	{H^n_{< \infty}(\BC^n \setminus \SA)} & {H^n(\BC^n \setminus \SA)} \\
	{H^n(X \setminus \OSA, Y \setminus (Y \cap \OSA))} & {H^n(X \setminus (Y \cup \OSA))}
	\arrow[from=1-1, to=1-2]
	\arrow["\simeq"', from=1-1, to=2-1]
	\arrow["\simeq"', from=1-2, to=2-2]
	\arrow[from=2-1, to=2-2]
\end{tikzcd}
\end{equation}
\end{lemma}

\begin{proof}
The diagram \eqref{eq2} commutes by standard properties of the six-functor formalism. The only subtle point is the use of \cite[Proposition 3.2]{Single-valued} and \cite[Proposition 1.7.5]{Ayoub}.
\end{proof}

As we observed in Remark \ref{remark}, we moreover have $H^n_{< \infty}(\BC^n \setminus \SA) \subseteq H^n(\BC^n \setminus \SA).$ This also implies that $H^n(X \setminus \OSA, Y \setminus (Y \cap \OSA)) \to H^n(X \setminus (Y \cup \OSA))$ is an inclusion.
Moreover, this latter map factors through $H^n(X \setminus \OSA),$ i.e., we have
\begin{equation}
	\label{adjunction}
	H^n(X \setminus \OSA, Y \setminus (Y \cap \OSA)) \to H^n(X \setminus \OSA) \to H^n(X \setminus (Y \cup \OSA)).
\end{equation}

\subsection{Vanishing of residue at infinity}
\label{section3.11}
In this section, we show that the cohomology at finite distance $H^n_{< \infty}(\BC^n \setminus \SA)$ is isomorphic to the joint kernel of the residue morphisms along the irreducible components of $Y$.

We consider the following setting: let $X$ be a smooth compactification of $\BC^n$ that is locally a product for $\SA,$ $Y:=X \setminus \BC^n$ be the divisor at infinity and $\OSA$ be the closure of $\SA$ in $X$. Moreover, let $(Y_i)_{i \in I}$ denote the irreducible components of $Y,$ and set
$$Y_i^\circ= Y_i \setminus \left( \bigcup_{i \ne j \in I} Y_i \cap (Y_j \cup \cap \OSA) \right).$$

\begin{theorem}
\label{theoremres}
Under the additional assumption that $Y \setminus (Y \cap \OSA)$ is a normal crossing divisor, we have the following isomorphism:
$$H^n(X \setminus \OSA, Y \setminus (\OSA \cup Y)) \simeq \ker \left( \bigoplus_{i \in I}\Res_{Y_i^\circ}\right),$$
where $\Res_{Y_i^\circ}: H^n(X \setminus (\OSA \cup Y)) \to H^{n-1}(Y_i^\circ)$ is the residue morphism.
\end{theorem}

\begin{remark}
Such compactifications of $\BC^n$ exist, and the following section will provide examples.
\end{remark}

Some proofs of this section and the following section rely on the theory of mixed Hodge structures introduced by Deligne in \cite{Hodge1,Hodge2,Hodge3} and mixed Hodge modules introduced by Saito in \cite{saito1990mixed}. We provide here a brief overview of the required properties and refer to these articles for further details and precise definitions.

The cohomology and homology of a variety, as well as the cohomology and homology of a pair of varieties, are endowed with a mixed Hodge structure (see \cite{Hodge1}, \cite{Hodge2}, and \cite{Hodge3}). A crucial property in this theory is that morphisms arising from geometry are morphisms of mixed Hodge structures, meaning that they are strictly compatible with both filtrations.

Here, we consider slightly more subtle cohomology groups. We require Saito's formalism of mixed Hodge modules as introduced in \cite{saito1990mixed}. This framework ensures that every object constructed via the six-functor formalism carries a mixed Hodge structure, generalising Deligne's construction.

In our setting, Saito's formalism implies that $H^n_{< \infty}(\BC^n \setminus \SA)$ is endowed with a mixed Hodge structure. Moreover, $H^n(\BC^n, \SA)$ carries a mixed Hodge structure arising from the filtrations. The inclusion $ H^n_{< \infty}(\BC^n \setminus \SA) \subseteq H^n(\BC^n \setminus \SA)$ is an inclusion of mixed Hodge structures.

\begin{proof}[Proof of Theorem \ref{theoremres}]
The long exact sequence in relative cohomology given by:
$$H^{n-1}(Y \setminus (\OSA \cup Y)) \to H^n(X \setminus \OSA, Y \setminus (\OSA \cup Y)) \to H^n(X \setminus \OSA) \to H^{n}(Y \setminus (\OSA \cup Y))$$ 
induces an isomorphism:
$$\Gr^W_{2n}H^n(X \setminus \OSA, Y \setminus (\OSA \cup Y)) \simeq \Gr^W_{2n}H^n(X \setminus \OSA).$$
Indeed, since $\dim(Y \setminus (Y \cup \OSA))=\dim(Y)=n-1,$ the cohomology groups $H^k(Y \setminus (Y \cap \OSA))$ do not have weight $2n$ for any $k$, i.e.,
$$\Gr^W_{2n} H^k(Y \setminus (Y \cup \OSA))=0 \text{ for all } k.$$
As we have seen in Remark \ref{remark}, we have the following inclusion $H^n_{< \infty}(\BC^n \setminus \SA) \hookrightarrow H^n(\BC^n \setminus \SA)$. The latter cohomology group is pure of weight $2n$ and type $(n,n)$ by \cite{brieskorn2006groupes}. This implies that $H^n_{< \infty}(\BC^n \setminus \SA)$ is pure of weight $2n$ and type $(n,n)$, and so is $H^n(X \setminus \OSA, Y \setminus (Y \cup \OSA))$ by Lemma \ref{lemmainclusion}. Therefore, we have:
$$H^n(X \setminus \OSA, Y \setminus (\OSA \cup Y)) = \Gr^W_{2n}H^n(X \setminus \OSA, Y \setminus (\OSA \cup Y)) \simeq \Gr^W_{2n}H^n(X \setminus \OSA).$$
For $J=\{i_1, \ldots, i_p\} \subseteq I,$ let $Y_J:=Y_{i_1} \cap \cdots \cap Y_{i_p}$ by $Y_J$ and set $$Y_J^\circ:=Y_J \setminus \left( \bigcup_{i \in I \setminus J}Y_J \cap \left(Y_i \cap \OSA\right)\right).$$ We consider the residue spectral sequence:
$$E_1^{p,q}=\bigoplus_{|J|=p}H^{-p+q}(Y_J^\circ)(-p) \implies H^{p+q}(X \setminus \OSA).$$
Taking the $\Gr^W_{2n}$-part of this spectral sequence yields a new spectral sequence whose $E_1$-page has only one row of non-vanishing terms (again because of the dimension of $Y_J^\circ$):
$$H^n(X \setminus (\OSA \cup Y)) {\xrightarrow{\bigoplus_{i \in I} \Res_{Y_i^\circ}}} \Gr^{W}_{2n}\left[\bigoplus_{i \in I} H^{n-1}(Y^\circ_i)(-1)\right] \to \Gr^{W}_{2n}\left[\bigoplus_{i \ne j \in I}H^{n-2}(Y^\circ_{ij})(-2)\right] \to \cdots.$$
It degenerates at the $E_1$-page. Therefore, we conclude that
\begin{equation*}
H^n(X \setminus \OSA, Y \setminus (\OSA \cup Y)) \simeq \Gr^{W}_{2n}H^n(X \setminus \OSA) \simeq \ker\left( \bigoplus_{i \in I}\Res_{Y_i^\circ}\right).
\qedhere
\end{equation*}
\end{proof}

\subsection{Partial wonderful compactifications at infinity}
\label{section2.2}

Theorems \ref{thmcharacterization} and \ref{theoremres} give two distinct descriptions of the cohomology at finite distance, each as the kernel of a map with source $H^n(\BC^n \setminus \SA).$
To relate the two descriptions, we make use of a specific compactification of $\BC^n$, referred to as partial wonderful compactifications at infinity.

In this section, we construct the compactification $X$ as a partial wonderful compactification. Wonderful compactifications are smooth compactifications of $\BC^n \setminus \SA$ for which the divisor $Y \cup \OSA$ is normal crossing. They were first introduced by de Concini and Procesi in \cite{de1995wonderful}. In contrast with the initial definition, we leave $\BC^n$ unchanged and only modify the divisor at infinity. The resulting compactification, which compactifies $\BC^n,$ however, is not a normal crossing compactification but locally a product for $\SA.$ We adapt the construction accordingly.
We follow the formulation of \cite[Section 3]{dupont2017relative} and refer to that work for further details and precise definitions. We work throughout in the setting of hypersurface arrangements. A hypersurface arrangement $\SB$ in a variety $X$ is a collection of hypersurfaces that is locally a hyperplane arrangement.

\begin{definition}
Let $\SA \subseteq \BC^n$ be a hyperplane arrangement. The dual of a stratum $S$ is the space of linear forms on $\BC^n$ that vanish on $S$. We denote it by $S^\perp.$ 

We say that the strata $S_1,\ldots, S_r$ intersect transversally if $S_1^\perp,\ldots, S_r^\perp$ are in direct sum in $\BC^n.$ We write $S_1 \pitchfork \cdots \pitchfork S_r.$

A decomposition of a stratum $S$ is an equality $S=S_1 \pitchfork \cdots \pitchfork S_r$ where $S_1, \ldots, S_r$ are strata of $\SA$ such that for every hyperplane $L:=\{f=0\}$ containing $S,$ then $L$ contains one of the $S_i$ for some $1 \leq i \leq r$. Equivalently, we say that $S=S_1 \pitchfork \cdots \pitchfork S_r$ is a decomposition of $S$ if $S=S_1^\perp \oplus \cdots \oplus S_r^\perp$ and for each equation $f$ defining a hyperplane, then $f \in S_i^\perp$ for some $1 \leq i \leq r$.

We say that $S$ is reducible if it admits a non-trivial decomposition, and irreducible if it does not.
\end{definition}

These definitions carry over to hypersurface arrangements. In this context, we require that the local hyperplane arrangement around each point satisfies the conditions.

For each stratum $S$ of $\SB$, there exists a unique decomposition of $S=S_1^\perp \oplus \cdots \oplus S_r^\perp$ into irreducible strata $S_1, \ldots, S_r$. Moreover, such a decomposition induces a product decomposition of $\SB_S \subseteq \BC^n$, the local arrangement around $S$: $\SB_S\simeq \SB_{S_1} \times \cdots \times \SB_{S_r}$ for $\SB_{S_1}, \ldots, \SB_{S_r}$ the local arrangements around $S_1, \ldots, S_r.$

In this section, we consider the following setting: we compactify $\BC^n$ to $\BP_\BC^n$ by adding a hyperplane at infinity $L_\infty$ and let $\tilde{\SA}$ denote the closure of $\SA$ in $\BP_\BC^n.$ Let $\SY=\{L_\infty\}$ denote the arrangement consisting of the single hyperplane $L_\infty.$ Moreover, let $\SG_0:=\{P_1, \ldots, P_r\}$ denote the set of all zero-dimensional strata of $\tilde{\SA} \cap L_\infty$ and let $\SG_{>0}^{\operatorname{irr}}$ be the set of irreducible strata $S$ of $\OSA \cup L_\infty$ that are fully contained in $L_\infty$ and satisfy $\dim(S) > 0.$
Alternatively, one can describe the set $\SG_0$ as the set of directions of the affine arrangement $\SA.$ Indeed, if $P \in L_\infty$ is a point at infinity, all hyperplanes $\tilde{L}$ passing through $P$ have the same direction $v.$ This notion of directions only depends on the affine arrangement $\SA.$

The idea of this lemma is to blow-up all
elements of $\SG \setminus \{L_\infty\}$ where $\SG:=\SG_0 \cup \SG_{>0}^{\operatorname{irr}}.$
At the end of the process, we obtain a variety $X$ which is locally a product for $\SA$. Moreover, for each irreducible stratum $S \in \SG_{>0}^{\operatorname{irr}}$ and each point $P \in \SG_0,$
we obtain an irreducible component of $Y:=X \setminus \BC^n$, denoted by $Y_S$ and $Y_P,$ respectively. These components are the strict transforms of the exceptional divisors arising from the blow-ups at $S$ and $P,$ respectively. Conversely, each irreducible component of $Y$ arises this way. We denote by $Y_\infty$ the component $Y_{L_\infty}$, which is the strict transform of $L_\infty.$

\begin{lemma}
\label{lemcomp}
We inductively define a sequence of complex manifolds $X^{(k)}$ and of two hypersurface arrangements $\tilde{\SA}^{(k)}$ and $\SY^{(k)}$ inside $X^{(k)}$ and of sets $\SG^{(k)}$ of strata of $\SA^{(k)} \cup \SY^{(k)}$ fully contained in $\SY^{(k)}$, via the following process:
\begin{enumerate}
	\item We start by setting $$X^{(0)}:=\BP_\BC^n, \ \tilde{\SA}^{(0)}:=\tilde{\SA}, \ \SY^{(0)}:=\SY, \text{ and } \SG^{(0)}:=\SG.$$
	\item Assume that, the variety $X^{(k)}$, the hypersurface arrangements $\tilde{\SA}^{(k)}$ and $\SY^{(k)}$ and the set $\SG^{(k)}$ are constructed and let $Z^{(k)}$ be a minimal element of $\SG^{(k)}$ of codimension greater than 2 in $X^{(k)}$. We define
	$$X^{(k+1)} \overset{\pi}{\to} X^{(k)}$$
	to be the blow-up of $X^{(k)}$ along $Z^{(k)}$. Moreover, we define
	$$\tilde{\SA}^{(k+1)}:=\{\tilde{L} \text{ the strict transform of } L \in \tilde{\SA}^{(k)}\},$$ for $E_{Z^{(k)}}$ the exceptional divisor of the blow-up along $Z^{(k)}$, $$\SY^{(k+1)}: =\{E_{Z^{(k)}}\} \cup \{\tilde{L} \text{ the strict transform of } L \in \SY^{(k)}\},$$ and $$\SG^{(k+1)}:=\{E_{Z^{(k)}}\} \cup \{\tilde{S} \text{ the strict transform of } S \in \SG^{(k)} \setminus \{Z^{(k)}\}\}.$$
\end{enumerate}
Then, after a finite number of steps, we get a variety $X:=X^{(\infty)}$ which is locally a product for $\SA$, together with two hypersurface arrangements, $\OSA:=\tilde{\SA}^{(\infty)}$ and $\overline{\SY}=\SY^{(\infty)},$ the latter being normal crossing. We also obtain a set $\overline{\SG}:=\SG^{(\infty)}$ of codimension-one strata of $\tilde{\SY}.$
\end{lemma}

\begin{definition}
\label{lemcomp1}
Such compactifications $X$ are referred to as partial wonderful compactifications at infinity. Moreover, we define the divisor at infinity $Y$ to be the union of all irreducible components of $\tilde{\SY}.$
\end{definition}

\begin{proof}
First, we notice that the condition to be locally a product for $\SA$ is equivalent to the following requirement: for each stratum $S$ of $\tilde{\SA}^{(\infty)} \cup \SY^{(\infty)},$ there exists a decomposition
$$S=S_1 \pitchfork \cdots \pitchfork S_{s_1} \pitchfork T_1 \pitchfork \cdots \pitchfork T_{s_2}$$
such that $S_1,\ldots, S_{s_1}$ are irreducible strata of $\tilde{\SA}^{(\infty)}$ and $T_1, \ldots, T_{s_2}$ are irreducible strata of $\SY^{(\infty)}.$ On the other hand, $\SY^{(\infty)}$ is normal crossing if and only if the only irreducible strata of $\SY^{(\infty)}$ are of codimension 1.

Let $\SI^{(k)}$ be the set of all irreducible strata of codimension greater than 2 of $\tilde{\SA}^{(k)} \cup \SY^{(k)}$ fully contained in a hypersurface of $\SY^{(k)}.$ According to the process, we start by blowing-up points, i.e., strict transforms of elements of $\SG_0.$ If there are $r'\leq r=|\SG_0|$ irreducible points in $\tilde{\SA} \cap \SY$, the elements in $\SG_0$ reduces the number $|\SI^{(0)}|$ by $r'$ by \cite[Lemma 3.5]{dupont2017relative}, i.e., $|\SI^{(r-1)}|=|\SI^{(0)}|-r'.$ For the elements of $\SG_{n-1}^{\operatorname{irr}}$, the same lemma shows that $|\SI^{(k+1)}|=|\SI^{(k)}|-1.$ In particular, after a finite number $t$ of steps, the only irreducible strata fully contained in a hypersurface of $\SY^{(t)}$ are hypersurfaces of $\SY^{(t)}.$ This implies that $\SY^{(t)}$ is normal crossing.

Let $S$ be a stratum admitting a decomposition $S=S_1 \pitchfork \cdots \pitchfork S_s$ where $S_1,\ldots, S_s$ are irreducible strata of $\tilde{\SA}^{(t)} \cup \SY^{(t)}.$ Then by definition of a decomposition, for each hypersurface $Y$ of $\SY^{(t)},$ there exists at most one irreducible stratum $S_i$ such that $Y \subseteq S_i,$ hence $Y=S_i.$ We conclude that all other irreducible strata cannot arise as intersections of hypersurfaces in $\SY^{(t)}.$ In particular, they are irreducible strata of $\tilde{\SA}^{(t)}.$
\end{proof}

\begin{remark}
At each step of the lemma, we choose a minimal irreducible stratum. However, \cite[Lemma 3.2]{Compactification}, and the first part of the proof of \cite[Theorem 1.3]{Compactification} ensures that the resulting variety is independent of the order of the blow-ups, provided they are performed in order of inclusion.
\end{remark}

At each step $k$ of the process, we choose a center $Z^{(k)} \in \SG^{(k)}$ of the blow-up and replace it with an irreducible stratum $Y^{(k)}$ of $\SY^{(k)}.$ In particular, each irreducible stratum of $Y,$ (i.e., irreducible component) is associated either with an irreducible stratum $S$ of $\tilde{\SA} \cup \SY$ fully contained in a hypersurface of $\SY$, (i.e., $S \in \SG_{n-1}^{\operatorname{irr}}$), or with a zero-dimensional stratum $P$ of $\tilde{\SA} \cap \SY$, (i.e., $P \in \SG_0$). We denote $Y_S \in \tilde{\SG_{n-1}^{\operatorname{irr}}},$ respectively $Y_P \in \tilde{\SG}$ the irreducible component of $\tilde{\SG}$ associated with $S$ and $P.$ Conversely, all irreducible components of $Y$ arise this way, i.e., $$Y=\bigcup_{S \in \SG}Y_S.$$
Moreover, for $P \in \SG_0,$ we define $$Y_{P}^\circ:= Y_P \setminus \left( \bigcup_{S \in \SG \setminus \{P\}} Y_{P} \cap \left(Y_S \cap \OSA\right) \right).$$

\subsection{Vanishing of residues at infinity for partial wonderful compactification at infinity}
\label{section2.22}
The main theorem of this section connects the descriptions given in Theorems \ref{thmcharacterization} and \ref{theoremres} and is proven independently. However, Corollary \ref{cor2} provides an alternative proof of the result of Theorem \ref{theoremres} for the case of partial wonderful compactifications at infinity.

\begin{theorem}
\label{lem1}
Let $X$, $Y$ and $\tilde{\SA}$ be as defined in Lemma \ref{lemcomp}.
Then,
\begin{enumerate}
	\item We have an isomorphism: $$\bigoplus_{P \in \SG_0} H^{n-1}(Y_{P}^\circ) \simeq H^{n-1}(\BC^n \setminus \SA);$$
	\item The following diagram commutes:
\[\begin{tikzcd}[row sep=huge,column sep=huge]
	{H^n(\BC^n \setminus \SA)} & &{H^{n-1}(\BC^n \setminus \SA)} \\
	{H^n(X \setminus \OSA, Y \setminus (Y \cup \OSA))} & &{\displaystyle{\bigoplus_{P \in \SG_0}H^{n-1}(Y^\circ_P)},}
	\arrow["\partial", from=1-1, to=1-3]
	\arrow["\simeq", from=1-1, to=2-1]
	\arrow["\simeq", from=1-3, to=2-3]
	\arrow["{\bigoplus_{P \in \SG_0}\Res_{Y_P^\circ}}", from=2-1, to=2-3]
\end{tikzcd}\]
where $\Res_{Y_P^\circ}: H^n(X \setminus \OSA, Y \setminus (Y \cup \OSA)) \to H^{n-1}(Y^\circ_P)$ is the residue morphism.
\end{enumerate}
\end{theorem}

\begin{proof}
All blow-ups of points are independent, so it suffices to show both results in $\BP_\BC^n$ after the first blow-up at a point $P$. We will perform several blow-ups after this first one according to Lemma \ref{lemcomp}. However, neither of the two cohomology groups changes under these blow-ups. Indeed, for any $k$, let $X^{(k+1)} \overset{\pi}{\to} X^{(k)}$ be the blow-up along some irreducible strata $S^{(k)}$ and $E_{S^{(k)}}$ denote its exceptional divisor.
Then, we have:
$$X^{(k+1)} \setminus (\SA^{(k+1)} \cup \SY^{(k+1)}) \simeq X^{(k+1)} \setminus (\SA^{(k)} \cup \SY^{(k)}),$$
Moreover, let $E_P$ be the exceptional divisor of the blow-up at $P$, $E_P^{(k)},$ $E_P^{(k+1)}:=\tilde{E}_P^{(k)},$ its strict transforms and $$\left(E_P^{(k)}\right)^\circ:=E_P \setminus \left( E_{P}^{(k)} \cap \left(\SY^{(k)} \cup \tilde{\SA}^{(k)}\right)\right) \text{ and } \left(E_P^{(k+1)}\right)^\circ:=E_P \setminus \left( E_{P}^{(k+1)} \cap \left(\SY^{(k+1)} \cup \tilde{\SA}^{(k+1)}\right)\right)$$ be the complement of all other irreducible components of $\SY^{(k)} \cup \tilde{\SA}^{(k)}$ in $E_P^{(k)}$, and $\SY^{(k+1)} \cup \tilde{\SA}^{(k+1)}$ in $E_P^{(k+1)}$, respectively.
Therefore, we have $$(E_P^{(k)})^\circ \simeq (E_P^{(k+1)})^\circ.$$
In particular, both the decomposition and the residue map remain unchanged in $X \setminus (Y \cup \OSA).$ 

\begin{enumerate}
\item We decompose $H^{n-1}(\BC^n \setminus \SA)$ using Brieskorn decomposition into:
$$H^{n-1}(\BC^n \setminus \SA)= \bigoplus_{l \in \SL_{n-1}(\SA)} A_{n-1}(\SA_l).$$
For a direction $v$ of $\BC^n$, define $\SA_v$, the subarrangement of $\SA$ consisting of all hyperplanes $L$ having direction $v.$ More precisely, for a direction vector $v=(v_1, \ldots, v_n),$ the arrangement $\SA_v$ consists of all hyperplanes $L=\{a_1x_1+ \cdots +a_nx_n+a_{n+1}=0\}$ such that $a_1v_1+ \cdots +a_nv_n=0.$ Equivalently, all hyperplanes with direction vector $v=(v_1, \ldots, v_n)$ pass through the same point $P_v=[v_1: \cdots: v_n:0] \in \tilde{\SA} \cap L_\infty$ at infinity. Figure \ref{figdir} illustrates, in dimension 2, the notion of direction. Moreover, we define $$\SA_v=\bigcup_{l \in \SL_{n-1}^v(\SA)} \SA_l,$$
where $\SL_{n-1}^v(\SA)$ denotes the set of lines of $\SA$ having direction $v.$
Equivalently, $\SL^v_{n-1}$ consists of all lines of $\tilde{\SA}$ meeting $P_v$ at infinity. Note that the arrangement $\SA_v$ is empty for all but finitely many $v$ and that, in dimension $n \geq 3,$ a given hyperplane may belong to several arrangements $\SA_v$. Each line of the arrangement has a unique direction, namely, the common direction of all hyperplanes containing it. Therefore, we can group together the corresponding summands in the Brieskorn decomposition to obtain: $$H^{n-1}(\BC^n \setminus \SA)= \bigoplus_{v} A_{n-1}(\SA_v),$$ where the direct sum ranges over all directions of $\BC^n.$
The arrangements $\SA_v$ are invariant under translations in the direction $v.$ Let $\BC v$ denotes the group of translations in the direction $v$ and $$\SA_v^{T}=\SA_v/\BC v \subseteq \BC^{n-1}$$ be the arrangement obtained by quotienting $\SA_v$ by $\BC v.$ Then, there is an isomorphism:
$$A_{n-1}(\SA_v^T) \simeq A_{n-1}(\SA_v).$$
We aim to show that this is moreover isomorphic to $H^{n-1}(E_{P_v}^\circ)$ where $E_{P_v}$ is the exceptional divisor of the blow-up at $P_v  \in \tilde{\SA} \cap L_\infty,$ the point associated with the direction $v$ and $$E_{P_v}^\circ=E_{P_v} \setminus \left( E_{P_v} \cap \left(L_\infty \cup \tilde{\SA}\right) \right).$$
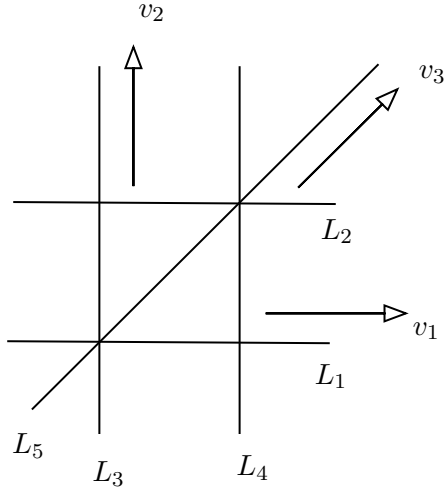
\begin{figure}
\tikzset{every picture/.style={line width=0.75pt}}

\begin{tikzpicture}[x=0.75pt,y=0.75pt,yscale=-1,xscale=1]

\draw    (51,41.67) -- (51,226.67) ;

\draw    (121,41.67) -- (121,226.67) ;

\draw    (166,181) -- (5,180) ;

\draw    (169,111) -- (8,110) ;

\draw    (17.33,214.33) -- (190.33,40.67) ;

\draw   (134,166) -- (193.33,166) -- (193.33,162) -- (204,166) -- (193.33,170) -- (193.33,166) -- (134,166) -- cycle ;

\draw   (150.25,102.75) -- (192.21,60.79) -- (189.38,57.97) -- (199.75,53.25) -- (195.03,63.62) -- (192.21,60.79) -- (150.25,102.75) -- cycle ;
 
\draw   (68,102) -- (68,42.67) -- (64,42.67) -- (68,32) -- (72,42.67) -- (68,42.67) -- (68,102) -- cycle ;

\draw (206,169) node [anchor=north west][inner sep=0.75pt]   [align=left] {$v_{1}$};

\draw (209,40) node [anchor=north west][inner sep=0.75pt]   [align=left] {$v_{3}$};

\draw (69,9) node [anchor=north west][inner sep=0.75pt]   [align=left] {$v_{2}$};

\draw (157,190) node [anchor=north west][inner sep=0.75pt]   [align=left] {$L_{1}$};

\draw (160,117) node [anchor=north west][inner sep=0.75pt]   [align=left] {$L_{2}$};

\draw (46,239) node [anchor=north west][inner sep=0.75pt]   [align=left] {$L_{3}$};

\draw (118,236) node [anchor=north west][inner sep=0.75pt]   [align=left] {$L_{4}$};

\draw (6,225) node [anchor=north west][inner sep=0.75pt]   [align=left] {$L_{5}$};

\end{tikzpicture}
\caption{The hyperplanes $L_1$ and $L_2$ have the same direction $v_1=(1,0)$, the hyperplanes $L_4$ and $L_5$ have direction $v_2=(0,1),$ while the hyperplane $L_5$ has direction $v_2=(1,1).$ In particular, $\SA_{v_1}=\{L_1, L_2\},$ $\SA_{v_2}=\{L_3\}$ and $\SA_{v_3}=\{L_4, L_5\}.$}
\label{figdir}
\end{figure}
Up to a linear change of coordinates, let us consider the direction $v=(1,0, \ldots, 0)$. Let $$L=\{a_1x_1+a_2x_2+ \cdots a_nx_n+a_{n+1}=0\} \in \SA_v$$ be a hyperplane in the direction $v,$ then $a_1=0.$ The arrangement $\SA^T_v$ consists exactly of those hyperplanes of the form
$$L^T=\{a_2x_2+ \cdots +a_n x_n+a_{n+1}=0\} \subseteq \BC^{n-1}.$$ All hyperplanes in $\SA_v$ intersect the hyperplane at infinity $L_\infty$ at the point $P=[1:0: \cdots: 0].$ We perform the blow-up at $P$ in the chart $U=\{x_1=1\} \subseteq \BP_\BC^n.$ In this chart, $L$ becomes $$\overline{L}=\{a_2x_2+\cdots+a_nx_n+a_{n+1}x_{n+1}=0\}.$$ There are $n$ local charts for the blow-up $\pi: \operatorname{Bl}_P(U) \to U$, indexed by $k= 2,\ldots,n+1.$ They are given by $$\pi(x_2,\ldots,x_{n+1})=(x_2x_k, \ldots, x_k, \ldots, x_{n+1}x_k).$$ We consider the index $k=n+1.$
Then, the exceptional divisor is $E_P=\{x_{n+1}=0\} \simeq \BC^{n-1}.$ The only hyperplane not visible in this affine chart is $\tilde{L}_\infty,$ the strict transform of $L_\infty.$ Hence, we view $E_P^\circ$ as the complement of an affine hyperplane arrangement. In this chart, the strict transform of the $\overline{L}$ becomes: 
$$\tilde{L}=\{a_2x_2+ \cdots+a_nx_n+a_{n+1}=0\} \subseteq \BC^{n-1}.$$ It has the same form as the hyperplane $L^T$ of $\SA_v^T,$ concluding the first part of the proof.

To fix notation for the next part, for $L_{i_1}=\{l_{i_1}=0\}, \ldots, L_{i_{n-1}}=\{l_{i_{n-1}}=0\} \in \SA$ hyperplanes such that $L_{i_1} \cap \cdots \cap L_{i_{n-1}}=l$ is a line in the direction $v$, let $$e_i \wedge \cdots \wedge e_{i_{n-1}}=d \log(l_{i_1}) \wedge \cdots \wedge d \log(l_{i_{n-1}})$$ be the generator of the Orlik--Solomon algebra associated with the line $l$.
We denote by $$f_{i_1} \wedge \cdots \wedge f_{i_{n-1}}=d \log(l_{i_1}) \wedge \cdots \wedge d \log(l_{i_{n-1}}),$$ the corresponding generator of the Orlik--Solomon algebra of $E_P^\circ.$ 
Equivalently, in projective coordinates, this generator becomes
$$\tilde{f}_{i_1} \wedge \cdots \wedge \tilde{f}_{i_{n-1}}+\sum_{k=1}^n(-1)^{n-k}\tilde{f}_{i_1} \wedge \cdots \wedge \widehat{\tilde{f}}_{i_k} \wedge \cdots \wedge \tilde{f}_{i_{n-1}} \wedge \tilde{f}_\infty$$
under the change of variables $f_i=\tilde{f}_i-\tilde{f}_\infty.$

\item For the second part of the proof, we again use Brieskorn decomposition:
$$H^n(\BC^n \setminus \SA)=\bigoplus_{Q \in \SL_n(\SA)} A_n(\SA_Q).$$
Let $e_I=e_{i_1} \wedge \cdots \wedge e_{i_n}$ be a generator corresponding to a point $Q=L_{i_1} \cap \cdots \cap L_{i_n} \in \BC^n.$
In projective coordinates, we have the change of variables
$$e_I=\tilde{e}_I + \sum_{j=1}^n (-1)^{n+1-j} \tilde{e}_{I\setminus \{i_j\}} \wedge \tilde{e}_\infty,$$
where $\tilde{e}_i-\tilde{e}_\infty$ are the generators of Orlik--Solomon algebra of the projective arrangement $\tilde{\SA}.$
For each $i_j \subseteq \{i_1, \ldots, i_n\},$ the line $l_{i_j}=L_{i_1} \cap \cdots \cap \widehat{L}_{i_j} \cap \cdots \cap L_{i_n}$ intersect $L_\infty$ at a point $P_{i_j}.$ We have
$$\partial(e_{i_1} \wedge \cdots \wedge e_{i_n})=\sum_{j \in I} (-1)^j{e}_{I \setminus \{i_j\}}.$$ 
By the first part of the proof, we can identify it with the element
\begin{equation}
\begin{cases}
	\label{element}
	(-1)^{n-j}(\tilde{f}_{I \setminus \{i_j\}} +\sum_{k=1, k \ne j}^n(-1)^{n+1-k} \tilde{f}_{I \setminus \{i_j, i_k\}} \wedge \tilde{f}_\infty) \in H^{n-1}(E_{P_{i_j}}^\circ); \\
	0 \text{ elsewhere.}
\end{cases}
\end{equation}
Here, $E_{P_{ij}}$ is the exceptional divisor of the blow-up at $P_{ij}$, $$E_{P_{i_j}}^\circ=E_{P_{i_j}} \setminus \left(E_{P_{i_j}} \cap \left(L_\infty \cap \tilde{\SA}\right)\right)$$ and $f_{I \setminus \{i_j, i_k\}}=\tilde{f}_{i_1} \wedge \cdots \wedge \widehat{\tilde{f}_{i_j}} \wedge \cdots\wedge \widehat{\tilde{f}_{i_k}} \wedge \cdots \wedge \tilde{f}_{i_{n}}$ We aim to show that $\Res_{E_{P_{i_j}}^\circ}(e_I)$ is the same element.

To this end, we consider the same setting as in the first part of the proof: let us assume that $P=[1:0: \cdots: 0].$ We consider the chart $U:=\{x_1=1\} \subseteq \BP_\BC^n.$ There are $n$ local charts for the blow-up $\pi: \operatorname{Bl}_P(U) \to U,$ indexed by $k=2, \ldots, n+1$. They are given by:
$$\pi(x_2,\ldots,x_{n+1})=(x_2x_k, \ldots, x_k, \ldots, x_{n+1}x_k).$$ Moreover, the exceptional divisor is $E_P=\{x_k=0\}.$

We start by showing that a differential $n$-form $\tilde{e}_{i_1} \wedge \cdots  \wedge \tilde{e}_{i_{n}}$ does not have residues along $E_P$ for $P \ne L_{i_1} \cap \cdots \cap L_{i_{n}}.$ Note that here $L_{i_n}$ may be $L_\infty$ (in which case $\tilde{e}_{i_n}$ is $\tilde{e}_\infty$). On $U,$ the form $\tilde{e}_{i_1} \wedge \cdots  \wedge \tilde{e}_{i_{n}}$ takes the form $$\frac{Adx_2 \wedge \cdots \wedge dx_{n+1}}{P(x_2, \ldots,x_{n+1})+Q(x_2, \ldots, x_{n+1})},$$
where $A$ is a constant, $P$ is a homogeneous polynomial of degree $n-1$ and $Q$ is a polynomial satisfying $\deg(Q)<n-1.$ The condition that $Q$ is non-zero is equivalent to the requirement that $H_{i_1} \cap \cdots \cap H_{i_{n}} \ne P.$ The pullback of this form along $\pi$ is
\begin{align*}
	(-1)^i\frac{Ax_i^{n-1}dx_1 \wedge \cdots \wedge dx_{k-1} \wedge dx_{k+1} \wedge \cdots \wedge  dx_{n+1}}{x_k^{n-1}P(x_1, \ldots, 1, \ldots x_{n+1})+\frac{1}{x_k}Q(x_kx_1, \ldots, x_i, \ldots, x_kx_{n+1})} \wedge \frac{dx_k}{x_k}.
\end{align*}
Its residue at $x_k=0$ vanishes
since the degree of $x_k$ of the polynomial $\frac{1}{x_k}Q$ is at most $n-2.$

Let now, $L_{i_2}, \ldots, L_{i_n}, L_\infty$ be $n$ hyperplanes whose intersection is $P=[1:0: \cdots: 0].$ We may assume up to a linear change of coordinates that $L_{i_j}=\{x_{j}=0\}$ and $L_\infty=\{x_{n+1}=0\}$ and hence $\tilde{e}_{i_j}=\frac{dx_{j}}{x_{j}}$ and $\tilde{e}_\infty=\frac{d{x_{n+1}}}{x_{n+1}}.$
On $U$, the form $\tilde{e}_{i_2} \wedge \cdots \wedge \tilde{e}_{\infty}$ is of the form $\frac{dx_2}{x_2} \wedge \cdots \wedge \frac{dx_{n+1}}{x_{n+1}}$ and its pullback along $\pi$ remains the same. Then, we have
\begin{align*}
	\Res_{x_k=0} (\frac{dx_2}{x_2} \wedge \cdots \wedge \frac{dx_k}{x_k} \wedge \cdots \wedge \frac{dx_{n+1}}{x_{n+1}})=\\
	=(-1)^{n+1-k}\frac{dx_2}{x_2} \wedge \cdots \wedge \frac{dx_{k-1}}{x_{k-1}} \wedge \frac{dx_{k+1}}{x_{k+1}}   \wedge \cdots \wedge \frac{dx_{n+1}}{x_{n+1}}.
\end{align*}
This expression is exactly $\tilde{f}_{i_2} \wedge \cdots \wedge \widehat{\tilde{f}}_{i_k} \wedge \cdots \wedge \tilde{f}_{i_n} \wedge \tilde{f}_{\infty}$ which is the only term of \eqref{element} that is visible in the chosen chart. This completes the proof. \qedhere
\end{enumerate}
\end{proof}

\begin{cor}
\label{cor2}
With the same notations as in Proposition \ref{lem1}, the following sequence is exact: $$0 \to H^n(X \setminus \OSA, Y \setminus (Y \cap \OSA)) {\xrightarrow{i}} H^n(X \setminus (Y\cup \OSA)) {{\xrightarrow{{\bigoplus_{S \in \SG}\Res_{Y_S^\circ}}}}} \bigoplus_{S \in \SG} H^{n-1}(Y_S^\circ).$$
\end{cor}

\begin{proof}
It follows from Lemma \ref{lemmainclusion} and Proposition \ref{lem1} that the following sequence is exact:
$$0 \to H^n(X \setminus \OSA, Y \setminus (Y \cap \OSA)) {\xrightarrow{i}} H^n(X \setminus (Y \cup \OSA)) {\xrightarrow{{\bigoplus_{P \in \SG_0} \Res_{Y_P^\circ}}}} \bigoplus_{P \in \SG_0} H^{n-1}(Y_P^\circ).$$
The map $i$ factors through $H^n(X \setminus \SA)$ by \eqref{adjunction}, thus $\im(i) \subseteq \ker\left(\bigoplus_{S \in \SG} \Res_{Y_S^\circ}\right).$ The reverse inclusion is given by the following diagram:
\[\begin{tikzcd}[row sep=huge,column sep=large]
	0 & {H^n(X \setminus \OSA, Y \setminus (Y \cap \OSA))} & {H^n(X \setminus \OSA)} && {\displaystyle{\bigoplus_{P \in \SG_0} H^{n-1}(Y_P^\circ)}} \\
	&&&& {\displaystyle{\bigoplus_{S \in \SG} H^{n-1}(Y_{S}^\circ)}}
	\arrow[from=1-1, to=1-2]
	\arrow["i",from=1-2, to=1-3]
	\arrow["{\bigoplus_{P \in \SG_0} Res_{Y_P}}"{pos=0.3}, from=1-3, to=1-5]
	\arrow["{\bigoplus_{S \in \SG} \Res_{Y_S^\circ}}"', from=1-3, to=2-5]
	\arrow["p"', two heads, from=2-5, to=1-5]
\end{tikzcd}\]
If $\omega \in \ker(\Res_{Y_S^\circ}),$ for all $S \in \SG$ then, in particular, $\Res_{Y_P^\circ}(\omega)=0$ for all $P \in \SG_0.$ Since the upper sequence is exact it follows that $\omega \in \im(i),$ concluding the proof.
\end{proof}

\subsection{Examples}
\label{Section3.4}

In this section, we will present a few examples of the previous theorem. We start with a two-dimensional example.

\begin{example}
\label{example1}
Consider the essential affine hyperplane arrangement $\SA=\{L_1,L_2,L_3,L_4,L_5\} \subseteq \BC^2$ where $$L_1=\{x_1=0\}, \ L_2=\{x_1-1=0\}, \ L_3=\{x_2=0\}, \ L_4=\{x_2-1=0\} \text{ and } L_5=\{x_1-x_2=0\}.$$ Then the homogeneous components of the Orlik--Solomon algebras are given by
$$A_1(\SA)= \BQ e_1 \oplus \BQ e_2 \oplus \BQ e_3 \oplus \BQ e_4 \oplus \BQ e_5$$
and
$$ A_2(\SA)= \frac{\BQ e_1 \wedge e_3 \oplus \BQ e_1 \wedge e_4 \oplus \BQ e_1 \wedge e_5 \oplus \BQ e_2 \wedge e_3 \oplus \BQ e_2 \wedge e_4 \oplus \BQ e_2 \wedge e_5, \oplus \BQ e_3 \wedge e_5, \oplus \BQ e_4 \wedge e_5}{(e_1 \wedge e_3-e_1 \wedge e_5+ e_3 \wedge e_5, e_2 \wedge e_4-e_2 \wedge e_5+ e_4 \wedge e_5)}.$$
In particular, the kernel of $\partial$ is $$\ker(\partial)=\BQ (e_1 \wedge e_4- e_1 \wedge e_5+e_4 \wedge e_5) \oplus \BQ (e_2 \wedge e_3- e_2 \wedge e_5+e_3 \wedge e_5).$$

Alternatively, we can view this arrangement in $\BP_\BC^2$ with homogeneous coordinates $[x_1:x_2:x_3]$ and the equations become
$$\tilde{L}_1=\{x_1=0\}, \ \tilde{L}_2=\{x_1-x_3=0\}, \ \tilde{L}_3=\{x_2=0\}, \ \tilde{L}_4=\{x_2-x_3=0\} \text{ and } \tilde{L}_5=\{x_1-x_2=0\}.$$
In that setting, the generator of the affine Orlik--Solomon algebra $e_i$ is identified with $\tilde{e_i} - \tilde{e}_\infty,$ the corresponding generator of the projective Orlik--Solomon algebra.
At infinity, the lines $\tilde{L}_1$ and $\tilde{L}_2$ meet at the point $P_{12}=[0:1:0]$ and the lines $L_3$ and $L_4$ meet at $P_{34}=[1:0:0]$. Lastly, the line $\tilde{L}_5$ meets $L_\infty$ at $P_5.$ In particular, the projective arrangement $\tilde{\SA} \cap L_\infty$ is not locally a product for $\SA$, so we perform a blow-up at each point to resolve the singularities. However, it is not necessary to resolve the singularities in $\BC^n.$ We first blow-up the point $P_{12},$ and we work in the chart $U:=\{x_2=1\}.$ The blow-up of $U$ in the intersecting chart is then given by
\begin{align*}
	\pi: \operatorname{Bl}_{P_{12}}U &\to U \\
	(x_1,x_3) & \mapsto [x_1x_3:x_3]
\end{align*}
Let $E_{P_{12}}=\{x_3=0\} \subseteq \BC^n$ denote the exceptional divisor of the blow-up at $P_{12}$ in this chart and set $E^\circ_{P_{12}}=E_{P_{12}} \setminus \left( E_{P_{12}} \cap \tilde{\SA} \right).$
The hyperplanes become
$$\overline{L}_1=\{x_1=0\}, \ \overline{L}_2=\{x_1-1=0\}, \ \overline{L}_3=\emptyset, \ \overline{L}_4=\{1-x_3=0\} \text{ and } \overline{L}_5=\{x_1x_3-1\} \subseteq \BC^2$$
and the intersections with $E_{P_{12}}$ are given by
$$\overline{L}_1 \cap E_{P_{12}}=\{x_1=0\}, \ \overline{L}_2 \cap E_{P_{12}}=\{x_1-1=0\}, \ \overline{L}_3 \cap E_{P_{12}}=\emptyset, \ \overline{L}_4 \cap E_{P_{12}}=\emptyset  \text{ and } \overline{L}_5 \cap E_{P_{12}}=\emptyset \subseteq \BC.$$

An illustration of the three situations, $\SA \subseteq \BC^2,$ $\tilde{\SA} \subseteq \BP_\BC^2$ and $\OSA \subseteq X$ is shown in Figure \ref{figrun1}.
For $f_1$, respectively $f_2$ the generators corresponding to $\overline{L}_1 \cap E_{P_{12}} \subseteq E_{P_{12}} \simeq \BC,$ and $\overline{L}_2 \cap E_{P_{12}} \subseteq E_{P_{12}} \simeq \BC$, respectively, we conclude that $$H^1(E_{P_{12}}^\circ)=\BQ f_1 \oplus \BQ f_2.$$
By considering the blow-ups at the points $P_{34}$ and $P_5$, we obtain the desired decomposition:
$$H^1(\BC^n \setminus \SA)=H^1(E_{P_{12}}^\circ) \oplus H^1(E_{P_{34}}^\circ) \oplus H^1(E_{P_{5}}^\circ).$$

Moreover, we compute the residues:
\begin{align*}
    \Res_{E_{P_{12}}^\circ}(e_1 \wedge e_3)=\Res_{E_{P_{12}}^\circ}(\tilde{e}_1 \wedge \tilde{e}_3-\tilde{e}_1 \wedge \tilde{e}_\infty+ \tilde{e}_3 \wedge \tilde{e}_\infty )= \\
    =\Res_{x_3=0}(-\frac{dx_1}{x_1} \wedge \frac{dx_3}{x_3})=-\frac{dx_1}{x_1}=-f_1.
\end{align*}
By symmetry, we conclude that $$\Res_{E_{P_{12}}^\circ}\oplus \Res_{E_{P_{34}}^\circ} \oplus \Res_{E_{P_{5}}^\circ} (e_1 \wedge e_3)=(-f_1,f_3,0)$$ and similarly,
\begin{align*}
\Res_{E_{P_{12}}^\circ}\oplus \Res_{E_{P_{34}}^\circ} \oplus \Res_{E_{P_{5}}^\circ}(e_1 \wedge e_4)=(-f_1,f_4,0), \\
\Res_{E_{P_{12}}^\circ}\oplus \Res_{E_{P_{34}}^\circ} \oplus \Res_{E_{P_{5}}^\circ}(e_1 \wedge e_5)=(-f_1,0,f_5),\\
\Res_{E_{P_{12}}^\circ}\oplus \Res_{E_{P_{34}}^\circ} \oplus \Res_{E_{P_{5}}^\circ}(e_2 \wedge e_3)=(-f_2,f_3,0), \\
\Res_{E_{P_{12}}^\circ}\oplus \Res_{E_{P_{34}}^\circ} \oplus \Res_{E_{P_{5}}^\circ}(e_2 \wedge e_4)=(-f_2,f_4,0), \\
\Res_{E_{P_{12}}^\circ}\oplus \Res_{E_{P_{34}}^\circ} \oplus \Res_{E_{P_{5}}^\circ}(e_2 \wedge e_5)=(-f_2,0,f_5), \\
\Res_{E_{P_{12}}^\circ}\oplus \Res_{E_{P_{34}}^\circ} \oplus \Res_{E_{P_{5}}^\circ}(e_3 \wedge e_5)=(0,-f_3,f_5), \\
\Res_{E_{P_{12}}^\circ}\oplus \Res_{E_{P_{34}}^\circ} \oplus \Res_{E_{P_{5}}^\circ}(e_4 \wedge e_5)=(0,-f_4,f_5). \\
\end{align*}
This is the same map as $\partial: A_2(\SA) \to A_1(\SA).$ In particular, we have
$$\ker(\partial) \simeq \BQ (e_1 \wedge e_4- e_1 \wedge e_5 +e_4 \wedge e_5) \oplus \BQ (e_2 \wedge e_3- e_2 \wedge e_5 +e_3 \wedge e_5) \simeq \ker(\Res_{E_{P_{12}}^\circ}\oplus \Res_{E_{P_{34}}^\circ} \oplus \Res_{E_{P_{5}}^\circ}).$$
\end{example}

The following example has no bounded regions, and hence the map $\partial$ has a trivial kernel. However, it is still an intersecting case as non-generic behaviors occur at infinity (see Figure \ref{fig2}).

\begin{example}
\label{example}
Consider the essential affine hyperplane arrangement $\SB$ given by $$K_1=\{x_1=0\}, K_2=\{x_1-1=0\}, K_3=\{x_2=0\}, K_4=\{x_3=0\}, K_5=\{x_2+x_3=0\}.$$
In projective coordinates, $\tilde{\SB}$: $$\tilde{K}_1=\{x_1=0\}, \tilde{K}_2=\{x_1-x_4=0\}, \tilde{K}_3=\{x_2=0\}, \tilde{K}_4=\{x_3=0\}, \tilde{K}_5=\{x_2+x_3=0\}.$$ Figure \ref{fig3d} is an illustration of the situation.
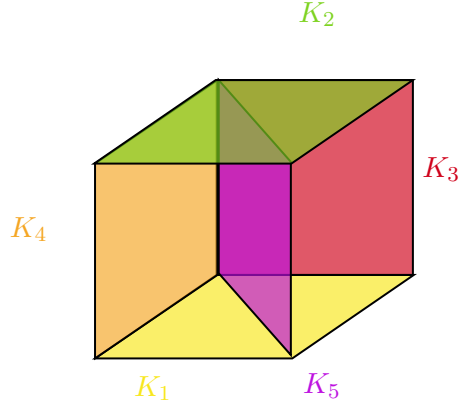
\begin{figure}
\tikzset{every picture/.style={line width=0.75pt}}

\begin{tikzpicture}[x=0.75pt,y=0.75pt,yscale=-2,xscale=2]

\draw  [fill={rgb, 255:red, 248; green, 231; blue, 28 }  ,fill opacity=0.66 ] (100.33,149.33) -- (149.56,149.33) -- (119.23,170.33) -- (70,170.33) -- cycle ;

\draw  [fill={rgb, 255:red, 208; green, 2; blue, 27 }  ,fill opacity=0.66 ] (100.33,100.33) -- (149.33,100.33) -- (149.33,149.33) -- (100.33,149.33) -- cycle ;

\draw  [fill={rgb, 255:red, 245; green, 166; blue, 35 }  ,fill opacity=0.66 ] (100.56,149.33) -- (100.56,100.33) -- (70,121.33) -- (70,170.33) -- cycle ;

\draw  [fill={rgb, 255:red, 189; green, 16; blue, 224 }  ,fill opacity=0.66 ] (118.87,120.89) -- (118.87,169.48) -- (100.87,148.92) -- (100.87,100.33) -- cycle ;
 
\draw  [fill={rgb, 255:red, 126; green, 211; blue, 33 }  ,fill opacity=0.66 ] (100.11,100.33) -- (149.33,100.33) -- (119,121.33) -- (69.77,121.33) -- cycle ;

\draw (79,174) node [anchor=north west][inner sep=0.75pt]  [color={rgb, 255:red, 248; green, 231; blue, 28 }  ,opacity=1 ] [align=left] {$\displaystyle K_{1}$};
\draw (120,80) node [anchor=north west][inner sep=0.75pt]  [color={rgb, 255:red, 126; green, 211; blue, 33 }  ,opacity=1 ] [align=left] {$\displaystyle K_{2}$};

\draw (48,134) node [anchor=north west][inner sep=0.75pt]  [color={rgb, 255:red, 245; green, 166; blue, 35 }  ,opacity=1 ] [align=left] {$\displaystyle K_{4}$};

\draw (151,119) node [anchor=north west][inner sep=0.75pt]  [color={rgb, 255:red, 208; green, 2; blue, 27 }  ,opacity=1 ] [align=left] {$\displaystyle K_{3}$};

\draw (121.23,173.33) node [anchor=north west][inner sep=0.75pt]  [color={rgb, 255:red, 189; green, 16; blue, 224 }  ,opacity=1 ] [align=left] {$\displaystyle K_{5}$};
\end{tikzpicture}
\caption{The arrangement ${\SB}$}
\label{fig3d}
\end{figure}
The homogeneous components $A_2(\SB)$ and $A_3(\SB)$ of the Orlik--Solomon algebra are given by:
$$\frac{\BQ e_1 \wedge e_3 \oplus \BQ e_1 \wedge e_4 \oplus \BQ e_1 \wedge e_5 \oplus \BQ e_2 \wedge e_3 \oplus \BQ e_2 \wedge e_4 \oplus \BQ e_2 \wedge e_5 \oplus \BQ e_3 \wedge e_4 \oplus \BQ e_3 \wedge e_5 \oplus \BQ e_4 \wedge e_5}{(e_3 \wedge e_4-e_3 \wedge e_5+e_4 \wedge e_5)},$$
respectively
$$\frac{ \BQ e_1 \wedge e_3 \wedge e_4 \oplus \BQ e_1 \wedge e_3 \wedge e_5 \oplus \BQ e_1 \wedge e_4 \wedge e_5 \oplus \BQ e_2 \wedge e_3 \wedge e_4 \oplus \BQ e_2 \wedge e_3 \wedge e_5 \oplus \BQ e_2 \wedge e_4 \wedge e_5}{(e_1 \wedge e_3 \wedge e_4-e_1 \wedge e_3 \wedge e_5+e_1 \wedge e_4 \wedge e_5, e_2 \wedge e_3 \wedge e_4-e_2 \wedge e_3 \wedge e_5+e_2 \wedge e_4 \wedge e_5)}.$$
One can verify that the kernel of $\partial: A_3(\SB) \to A_2(\SB)$ is trivial, which is consistent with the fact that there are no bounded regions.

As shown in picture \ref{fig2}, the hyperplanes $\tilde{K}_1, \tilde{K}_2$ intersect $L_\infty$ along the same line.
The intersection points of $L_\infty$ with the lines of the arrangements are given by:
\begin{align*}
	P_{123}:=\tilde{K}_1 \cap \tilde{K}_2 \cap \tilde{K}_3 \cap L_\infty=[0:0:1:0], \ P_{124}:=\tilde{K}_1 \cap \tilde{K}_2 \cap \tilde{K}_4 \cap L_\infty=[0:1:0:0], \\
	P_{125}:=\tilde{K}_1 \cap \tilde{K}_2 \cap \tilde{K}_5 \cap L_\infty=[0:1:-1:0] \text{ and } P_{345}:=\tilde{K}_3 \cap \tilde{K}_4 \cap \tilde{K}_5 \cap L_\infty=[1:0:0:0].
\end{align*}
Equivalently, for the vectors $v_{123}:=(0,0,1),$ $v_{124}:=(0,1,0),$ $v_{125}:=(0,1,-1),$ and $v_{345}:=(1,0,0),$ we have $$\SB_{v_{123}}=\{K_1, K_2, K_3\}, \ \SB_{v_{124}}=\{K_1, K_2, K_4\}, \ \SB_{v_{125}}=\{K_1, K_2, K_5\} \text{ and } \SB_{v_{345}}=\{K_3, K_4, K_5\}.$$ These subarrangements induce the desired decomposition:
$$A_2(\SB)= A_2(\SB_{v_{123}}) \oplus A_2(\SB_{v_{124}}) \oplus A_2(\SB_{v_{125}}) \oplus A_2(\SB_{v_{345}}).$$

The blow-ups at $P_{123},P_{124}$ and $P_{125},$ have a different structure than the blow-up at $P_{345}.$
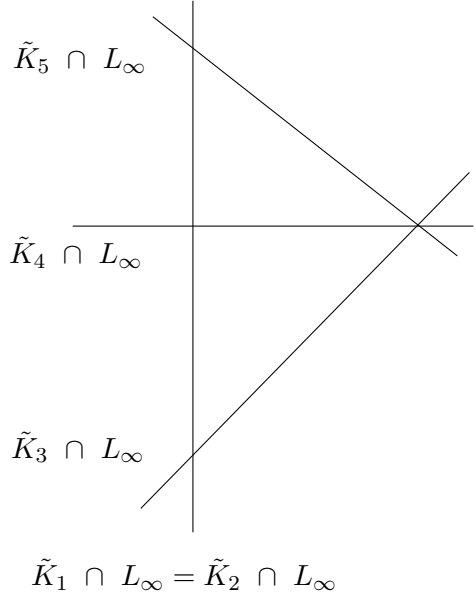
\begin{figure}
\begin{tikzpicture}[x=0.75pt,y=0.75pt,yscale=-1,xscale=1]

\draw    (286,9.67) -- (286,277) ;

\draw    (266,17.67) -- (418,138) ;

\draw    (424,96) -- (260,265) ;
 
\draw    (226,123) -- (426,123) ;

\draw (204,289) node [anchor=north west][inner sep=0.75pt]   [align=left] {$\displaystyle \tilde{K}_{1} \ \cap \ L_{\infty } =\tilde{K}_{2} \ \cap \ L_{\infty }$};

\draw (194,225) node [anchor=north west][inner sep=0.75pt]   [align=left] {$\displaystyle \tilde{K}_{3} \ \cap \ L_{\infty }$};

\draw (193,127) node [anchor=north west][inner sep=0.75pt]   [align=left] {$\displaystyle \tilde{K}_{4} \ \cap \ L_{\infty }$};

\draw (195,29) node [anchor=north west][inner sep=0.75pt]   [align=left] {$\displaystyle \tilde{K}_{5} \ \cap \ L_{\infty }$};

\end{tikzpicture}
\caption{The arrangement $\tilde{\SB} \cap L_\infty$ in $L_\infty$}
\label{fig2}
\end{figure}
One can check that on the exceptional divisor $E_{P_{123}}$ of the blow-up at $P_{123},$ we view two parallel lines and a line transverse to them, as shown in Figure \ref{fig3}. Let $f_1 \wedge f_3$ and $f_2 \wedge f_3$ denote the generators of the homogeneous components of the Orlik--Solomon algebra $H^2(E_{P_{123}}^\circ),$ i.e.,
$$H^2(E_{P_{123}}^\circ)=\BQ f_1 \wedge f_3 \oplus \BQ f_2 \wedge f_3.$$
The same configuration occurs for the exceptional divisors $E_{P_{124}}$ and $E_{P_{125}}$ of the blow-ups at $P_{124},$ and $P_{125},$ respectively.
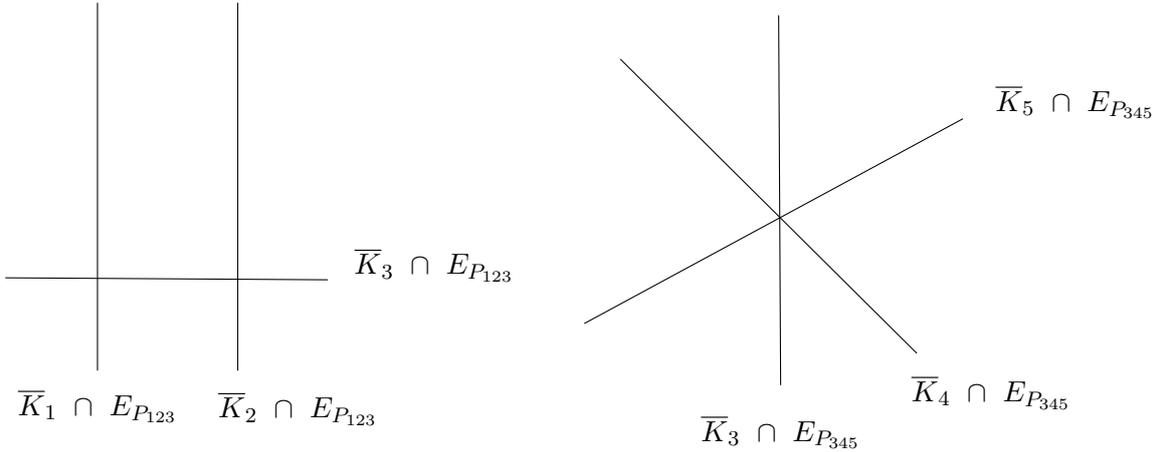
\begin{figure}
\begin{tikzpicture}[x=0.75pt,y=0.75pt,yscale=-1,xscale=1]

\draw    (100,23.67) -- (100,208.67) ;
 
\draw    (170,23.67) -- (170,208.67) ;

\draw    (215,163) -- (54,162) ;

\draw    (361,52) -- (509,200) ;

\draw    (343,185) -- (532,82) ;

\draw    (440,30) -- (441,216) ;

\draw (59,217) node [anchor=north west][inner sep=0.75pt]   [align=left] {$\displaystyle \overline{K}_{1} \ \cap \ E_{P_{123}}$};

\draw (159,218.17) node [anchor=north west][inner sep=0.75pt]   [align=left] {$\displaystyle \overline{K}_{2} \ \cap \ E_{P_{123}}$};

\draw (227,146) node [anchor=north west][inner sep=0.75pt]   [align=left] {$\displaystyle \overline{K}_{3} \ \cap \ E_{P_{123}}$};

\draw (400,230) node [anchor=north west][inner sep=0.75pt]   [align=left] {$\displaystyle \overline{K}_{3} \ \cap \ E_{P_{345}}$};

\draw (505,210) node [anchor=north west][inner sep=0.75pt]   [align=left] {$\displaystyle \overline{K}_{4} \ \cap \ E_{P_{345}}$};

\draw (547,64) node [anchor=north west][inner sep=0.75pt]   [align=left] {$\displaystyle \overline{K}_{5} \ \cap \ E_{P_{345}}$};

\end{tikzpicture}
\caption{From left to right, the arrangements $\overline{\SB} \cap E_{P_{123}}$, and $\overline{\SB} \cap E_{P_{345}}.$}
\label{fig3}
\end{figure}
On the exceptional divisor $E_{P_{345}}$ of the blow-up at $P_{345},$ we instead observe three lines meeting in a single point as illustrated in Figure \ref{fig3}. In particular, $$H^2(E_{P_{345}}^\circ)=\frac{\BQ f_3 \wedge f_4 \oplus \BQ f_3 \wedge f_5 \oplus \BQ f_4 \wedge f_5}{(f_3 \wedge f_4-f_3 \wedge f_5+f_4 \wedge f_5)}.$$
We obtain the required decomposition:
$$H^2(\BC^3 \setminus \SA)=H^2(E_{P_{123}}^\circ) \oplus H^2(E_{P_{124}}^\circ) \oplus H^2(E_{P_{125}}^\circ) \oplus H^2(E_{P_{345}}^\circ).$$
On the residue side, one can, for example, verify that
\begin{align*}
	Res_{E_{P_{123}}^\circ}(e_1 \wedge e_3 \wedge e_4)=-f_1 \wedge f_3,\\
	Res_{E_{P_{124}}^\circ}(e_1 \wedge e_3 \wedge e_4)=f_1 \wedge f_4,\\
	Res_{E_{P_{345}}^\circ}(e_1 \wedge e_3 \wedge e_4)=-f_3 \wedge f_4.
\end{align*}
and all other residues of $e_1 \wedge e_3 \wedge e_4$ vanish. We compare this map with $\partial.$
\end{example}

Note that in this example, the method of Lemma \ref{lemcomp} does not produce a normal crossing divisor $\overline{\SB} \cup \overline{\SY}.$ The difficulty arises from the component $Y_{P_{345}}.$ Nevertheless, the construction remains locally a product for $\SB$. 

\section{Canonical forms}
\label{Section4}
\subsection{Definition of a canonical form}
\label{sec4.1}
The notion of positive geometry was first introduced in \cite{PosGeom} in the context of particle physics. This notion relies on the existence and uniqueness of a canonical form associated with the given positive geometry. In loc.~cit., the definition is constructive (see also \cite{Lam} for further details). In \cite{brown2025positive}, the authors provide a Hodge-theoretic perspective on positive geometries and canonical forms. This article addresses the case where the ambient space $X$ is compact. For an affine hyperplane arrangement, however, the ambient space $\BC^n$ is non-compact. In this section, we recast the definition of canonical forms in the setting of essential affine hyperplane arrangements. A motivating example for studying the affine setting was to understand in which space the integrands of the cosmological correlators, as defined in \cite{cosmological}, naturally take values.

As noted in \cite[The Compactness Assumption]{brown2025positive}, one can compactify $\BC^n$ and work with the Poincaré duals
$$H_n(X, \OSA) \simeq H^n(X \setminus \OSA),$$
where $X$ is a compactification of $\BC^n$ that is locally a product for $\SA$ and $\OSA$ is the closure of $\SA$ in $X$ as in Section \ref{Section3}. 
However, when the compactification is not part of the given data, compactifying might require performing several blow-ups, and it is often convenient to avoid doing so. We present here an alternative definition of the canonical map, which turns out to be equivalent but is expressed only in terms of the affine hyperplane arrangement.

As we have seen in the proof of Theorem \ref{theoremres}, the mixed Hodge structure on $H_n(\BC^n, \SA)$ is pure of weight zero and type $(0,0)$. In particular,
\begin{align}
\label{equ}
F^0 H_n(\BC^n, \SA)_{\BC}=H_n(\BC^n, \SA)_{\BC}=\operatorname{gr}^W_0 H_n(\BC^n, \SA)_{\BC},
\end{align}
and by Poincaré duality, it follows that $H^n_{< \infty}(\BC^n \setminus \SA)_{\BC}$ is pure of weight $2n$ and type $(n,n)$.

\begin{remark}
The Hodge-theoretic definition of canonical forms requires the notion of genus. In the non-compact case, one could define the genus as
$$g(\BC^n, \SA)=\sum_{p<0}h^{p,0}(H_n(\BC^n, \SA)).$$
However, this quantity vanishes for all pairs $(\BC^n, \SA),$ so it plays no role in our setting.
\end{remark}

\begin{definition}
We define the canonical map $\can$ by Poincaré duality:
\begin{align*}
    \operatorname{can}:  H_n(\BC^n, \SA)& \to   H^n_{< \infty}(\BC^n \setminus \SA)_{\BC}\\
     \sigma & \mapsto  \omega_{\sigma}.
\end{align*}
\end{definition}

Note that there is a natural inclusion given by
$$H^n_{< \infty}(\BC^n \setminus \SA)_{\BC} \subseteq H^n(\BC^n, \SA)_\BC= \Omega_{\operatorname{log}}^n(\BC^n \setminus \SA),$$
where the last equality comes from \cite[Equation (21)]{brown2025positive}.
Therefore, as expected, every canonical form is logarithmic. Moreover, the description of Section \ref{section2} shows that the canonical forms are precisely the logarithmic forms whose residues along the divisor at infinity vanish. 

However, this new definition of the canonical map $\can$ agrees with the canonical map obtained after compactifying $\BC^n$. Indeed, let $X$ be a compactification which is locally a product for $\SA,$ and let $Y:= X \setminus \BC^n$ be its divisor at infinity and $\OSA$ the closure of $\SA$ in $X$.
We established in Theorem \ref{theoremres}, that $$\operatorname{gr}^W_{2n}H^n_{< \infty}(\BC^n \setminus \SA)\simeq \operatorname{gr}^W_{2n}H^n(X \setminus \OSA, Y \setminus (Y \cap \OSA)) \simeq \operatorname{gr}^W_{2n}H^n(X \setminus \OSA)$$
and dually, $$\operatorname{gr}^W_{0}H_n(\BC^n, \SA)\simeq \operatorname{gr}^W_{0}H_n(X, \OSA).$$
In particular, Poincaré duality yields a commutative diagram:
\[\begin{tikzcd}
	{H_n(\BC^n, \SA)} & {H^n_{< \infty}(\BC^n \setminus \SA)_{\BC}} \\
	{H_n(X, \OSA)} & {H^n(X \setminus \OSA)_\BC}
	\arrow["{\operatorname{can}}", from=1-1, to=1-2]
	\arrow[from=1-1, to=2-1]
	\arrow[from=1-2, to=2-2]
	\arrow["{\operatorname{can}}", from=2-1, to=2-2]
\end{tikzcd}\]

\subsection{Formula for the canonical form}
\label{sec4.2}

In this section, we rederive the canonical form formula of \cite[Proposition 6.7]{brown2025positive} in the setting of essential affine hyperplane arrangements. The results are similar to those of Section 6 in loc.~cit. We start with a technical lemma, which serves as the analogue of \cite[Proposition 2.14]{brown2025positive}.

\begin{lemma}
\label{lem2.14}
Let $\SA=\{L_1, \ldots, L_m\} \subseteq \BC^n$ be an essential affine hyperplane arrangement and let $L_j \in \SA$. Set $$L_j^\circ=L_j \setminus \left(\bigcup_{i \ne j=1}^m L_i \cap L_j \right).$$ Then, the following diagram commutes:
\begin{equation}
\label{diagram}
\begin{tikzcd}
	{H_n(\BC^n, \SA)_\BC} & {H^n_{< \infty}(\BC^n \setminus \SA)_{\BC}} & {H^n(\BC^n \setminus \SA)_{\BC}}\\
	{H_{n-1}(L_j, L_j \setminus L_{j}^\circ)_\BC} & {H^{n-1}_{< \infty}(L_j^\circ)_{\BC}} & {H^{n-1}(L_j^\circ)_{\BC}}
	\arrow["\can", from=1-1, to=1-2]
	\arrow["{ \partial_{L_j \setminus L_j^\circ}}"', from=1-1, to=2-1]
	\arrow["{ \Res_{L_j^\circ}}", from=1-2, to=2-2]
	\arrow["\can", from=2-1, to=2-2]
    \arrow[hook, from=1-2, to=1-3]
    \arrow[hook, from=2-2, to=2-3]
    \arrow["{  \Res_{L_j^\circ}}", from=1-3, to=2-3]
\end{tikzcd}
\end{equation}
\end{lemma}

\begin{proof}
Each vertical map arises from a standard distinguished triangle. The leftmost and middle vertical maps are related by Poincaré duality and hence the corresponding square commutes up to a sign. A variant of \cite[Proposition 4.12]{Single-valued} then shows the strict commutativity under Poincaré duality.

The commutativity of the rightmost square follows by applying the natural morphism of functors $Ra_! \to Ra_*$ to the corresponding distinguished triangles.
\end{proof}

The following proposition is the affine analogue of \cite[Proposition 6.7]{brown2025positive} and will be formulated in a very similar manner. Let us fix an ordering on the hyperplanes, say $\SA=\{L_1, \ldots, L_m\}$. For $I=\{i_1, \ldots, i_n\} \subseteq \{1, \ldots, m\},$ a set of indices, let $L_I=L_{i_1} \cap \cdots \cap L_{i_n}.$ We define the iterated boundary map
$$\partial_{I}=\partial_{L_{i_1} \cap \cdots \cap L_{i_n}} \circ \partial_{L_{i_2} \cap \cdots \cap L_{i_n}} \circ \cdots \circ \partial_{L_{i_n}}: H_n(\BC^n, \SA) \to H_0(L_I) \simeq \BQ.$$
Moreover, we define the logarithmic form associated with $I$:
$$\omega_{I}=\omega_{i_1} \wedge \cdots \wedge \omega_{i_n} \in H^n(\BC^n \setminus \SA)$$
where $L_{i_k}=\{f_{i_k}=0\}$ and $\omega_{i_k}=d\log(f_{i_k})$.

\begin{prop}
\label{prop6.7}
For $\sigma \in H_n (\BC^n, \SA)$, the associated canonical form is 
\begin{equation}
	\label{eqfin}
\omega_\sigma= \sum_{I \text{ nbc set, }
|I|=n} \partial_I (\sigma)\omega_I
\end{equation}
where the sum runs over all non-broken circuit (nbc) sets of $\SA$ as defined in \cite[Definition 3.35]{Orlik-Terao}.
\end{prop}

\begin{proof}
As in \cite[Proposition 6.7]{brown2025positive}, we iterate Lemma \ref{lem2.14} to find the following commutative diagram:
\begin{equation}
\label{diag2}
	\begin{tikzcd}
	{H_n(\BC^n, \SA)_\BC} & {H^n_{< \infty}(\BC^n \setminus\SA)_\BC} & {H^n(\BC^n \setminus \SA)_\BC} \\
	{\bigoplus_{I \text{ nbc set,} |I|=n}H_0(L_I)_\BC} & {\bigoplus_{I \text{ nbc set}, |I|=n}H^0_{< \infty}(L_I)_\BC} & {\bigoplus_{I \text{ nbc set}, |I|=n}H^0(L_I)_\BC}
	\arrow[from=1-1, to=1-2]
	\arrow["\partial_I",from=1-1, to=2-1]
	\arrow[from=1-2, to=1-3]
	\arrow["\Res_I",from=1-3, to=2-3]
	\arrow[from=2-1, to=2-2]
	\arrow[from=2-2, to=2-3]
\end{tikzcd}
\end{equation}
As in \cite[Proposition 6.7]{brown2025positive}, \cite[Proposition 3.6]{Szenes} and \cite[Theorem 3.55]{Orlik-Terao} show that the rightmost vertical map of the diagram \eqref{diag2} is an isomorphism whose inverse sends the generator $1 \in H^0(L_I)_\BC$, to the logarithmic form $\omega_I \in H^n(\BC^n \setminus \SA)$. In particular, since we have the natural inclusion $$H^n_{< \infty}(\BC^n \setminus \SA) \hookrightarrow H^n(\BC^n \setminus \SA).$$
We conclude that
\begin{equation*}
	\omega_\sigma = \sum_{I \text{ nbc set}, \ |I |=n} \partial_I(\sigma)\omega_I \in H^n_{< \infty}(\BC^n \setminus \SA). \qedhere
\end{equation*}
\end{proof}

\begin{remark}
Proposition \ref{prop6.7} produces forms $\omega$ that lie in $\ker(\partial).$ However, this is not immediate from the formula \eqref{eqfin} that $\omega_\sigma$ satisfies $\partial(\omega_\sigma).$
\end{remark}

\begin{example}
Our initial motivation for this work was to understand the integrand of the cosmological correlators.
To illustrate this, we consider the case of the two-site graph. Following the combinatorial method of \cite{cosmological} this graph gives rise to the line arrangement $\SA \subseteq \BC^2$ defined by $$L_1=\{x_1+X_1+Y=0\}, \ L_2=\{x_2+X_2+Y=0\} \text{ and } L_3=\{x_1+x_2+X_1+X_2=0\},$$ where $x_1, x_2$ are variables in $\BC^2$ and $X_1,X_2$ and $Y$ are constants. We have one bounded region (a triangle) and every $2$-element subset of $\{1,2,3\}$ is a non-broken circuit set. In particular, the canonical form associated with this region is
$$e_1 \wedge e_2 - e_1 \wedge e_3 +e_2 \wedge e_3.$$
In coordinates, this becomes
$$\frac{-2Y dx_1 \wedge dx_2}{(x_1+X_1+Y)(x_2+X_2+Y)(x_1+x_2+X_1+X_2)},$$
which, up to a multiplicative constant, coincides with the two-site cosmological correlator.
\end{example}

Further examples of cosmological correlators will be addressed in a forthcoming work.

\newpage

\appendix

\section{Interpretation in terms of smooth forms}
\label{section3.0}
For this appendix, we temporarily replace the constant sheaf $\BQ_{\BC^n \setminus \SA}$ with $\BR_{\BC^n \setminus \SA}.$ Our goal is to interpret the group $$H^n_{< \infty}(\BC^n, \SA)_{\BR} := H^n_c(\BC^n, Rj_* \BR_{\BC^n \setminus \SA})$$ in terms of real smooth forms $\omega$ for which there exists a compact $K \subseteq \BC^n$ such that $\omega$ vanishes outside $K$. This interpretation motivates the terminology of cohomology at finite distance for $H^n_{< \infty}(\BC^n \setminus \SA)$.

First, we need two technical lemmas using the functors $a_!$ and $j_*,$ where $a: \BC^n \to \{\star\}$ is the map to a point and $j: \BC^n \setminus \SA \to \BC^n$ is the open immersion.

\begin{lemma}
\label{lemma1}
There is a natural isomorphism of derived functors $$R(a_! \circ j_*) \simeq Ra_! \circ Rj_*.$$
\end{lemma}

\begin{proof}
By \cite[Proposition 2.58]{Huybrechts} and \cite[Remark 2.59]{Huybrechts}, it suffices to verify that for every injective sheaves $\SI$ on $\BC^n \setminus \SA$, the sheaf $j_* \SI$ is $a_!$-acyclic. Let $\SI$ be an injective sheaf, then $\SI$ is flabby \cite[Lemma 2.4]{Hartshorne}. Recall that a sheaf $\SF$ on $\BC^n \setminus \SA$ is flabby, if for every open $U \subseteq \BC^n \setminus \SA$, the restriction map $\Gamma(\BC^n \setminus \SA, \SF) \to \Gamma(U,\SF),$
is surjective. If a sheaf $\SF$ is flabby on $\BC^n \setminus \SA,$ then the sheaf $j_* \SF$ is also flabby on $\BC^n$. Indeed, let $U \subseteq \BC^n$ be an open subspace, then the map $\Gamma(\BC^n, j_*\SF) \to \Gamma(U,j_*\SF),$ or equivalently the map $\Gamma(j^{-1}(\BC^n), \SF) \to \Gamma(j^{-1}(U),\SF)$ is surjective. We conclude that for every $\SI$ injective, $j_* \SI$ is flabby, hence $c$-soft by \cite[Proposition 4.3.3]{kashiwaraintroduction}. Therefore, by \cite[Proposition 4.3.9]{kashiwaraintroduction}, $j_*\SI$ is $a_!$-acyclic. 
\end{proof}

\begin{lemma}
\label{lemma2}
Let $\SE^\bullet_{\BC^n \setminus \SA}$ denote the complex of sheaves of smooth real differential forms on $\BC^n \setminus \SA.$ Then, 
\begin{enumerate}
	\item The complex $\SE^\bullet_{\BC^n \setminus \SA}$ is $j_*$-acyclic;
	\item The complex $j_* \SE^\bullet_{\BC^n \setminus \SA}$ is $a_!$-acyclic.
\end{enumerate}
\end{lemma}

\begin{proof}
\begin{enumerate}
	\item It is well-known that the sheaf $\SE^\bullet_{U}$ is $a_*$-acyclic for every open $U \subseteq \BC^n,$ (see \cite[Section II.3.7]{godement}) that is,
	\begin{equation}
		\label{equationRj}
	R^ia_* \SE^\bullet_{U}=H^i(U, \SE^\bullet_U)=0.
	\end{equation} Recall that $j_*$-acyclicity for $\SE^\bullet_{\BC^n \setminus \SA}$ means that $$R^ij_* \SE^\bullet_{\BC^n \setminus \SA}(U)=0 \text{ for every open } U \subseteq \BC^n \setminus \SA.$$ The sheaf $R^ij_* \SE^\bullet_{\BC^n \setminus \SA}$ is the sheafification of the presheaf
	$$V \mapsto H^i(V \cap (\BC^n \setminus \SA), \SE^\bullet_{|V \cap (\BC^n \setminus \SA)}).$$
	These groups vanish by \eqref{equationRj}. Consequently, the complex $\SE^\bullet_{\BC^n \setminus \SA}$ is $j_*$-acyclic.
	\item The sheaves $\SE^\bullet_{\BC^n \setminus \SA}$ are fine by \cite[Section II.3.7]{godement} and fine sheaves are soft by \cite[Section II.3.7]{godement}. Recall that a sheaf $\SF$ on $\BC^n \setminus \SA$ is soft, if for every closed $C \subseteq \BC^n \setminus \SA$, the restriction map
	$\Gamma(\BC^n \setminus \SA, \SF) \to \Gamma(C,\SF),$
	is surjective. If a sheaf $\SF$ is soft on $\BC^n \setminus \SA,$ then the sheaf $j_* \SF$ is also soft. Indeed, let $D \subseteq \BC^n$ be a closed subspace. The map $\Gamma(\BC^n, j^*\SF) \to \Gamma(D,j_*\SF),$ or equivalently, the map $\Gamma(j^{-1}(\BC^n), \SF) \to \Gamma(j^{-1}(D),\SF)$ is surjective. We conclude that $j_* \SE^\bullet_{\BC^n \setminus \SA}$ is soft. Since by definition, all soft sheaves are $c$-soft, it follows that $j_* \SE^\bullet_{\BC^n \setminus \SA}$ is $c$-soft. Therefore, by \cite[Proposition 4.3.9]{kashiwaraintroduction}, the complex $j_* \SE^\bullet_{\BC^n \setminus \SA}$ is $a_!$- acyclic. \qedhere
\end{enumerate}
\end{proof}

\begin{prop}
\label{lemmaresolution}
For $\BR_{\BC^n \setminus \SA}$, the constant sheaf on $\BC^n \setminus \SA$, we have the following isomorphism:
$$H^k_c(\BC^n, Rj_* \BR_{\BC^n \setminus \SA}) \simeq H^k(\Gamma_c(\BC^n, j_* \SE^\bullet_{\BC^n \setminus \SA})).$$
\end{prop}

\begin{proof}
By Lemmas \ref{lemma1} and \ref{lemma2}, the complex $\SE_{\BC^n \setminus \SA}^\bullet$ is a $(a_! \circ j_*)$-acyclic resolution of $\BR_{\BC^n \setminus \SA}.$ Therefore, we have the following isomorphism $$H^k_c(\BC^n, Rj_* \BR_{\BC^n \setminus \SA}) \simeq H^k(\BC^n, a_!j_* \SE^\bullet_{\BC^n \setminus \SA}).$$ Since $j_* \SE_{\BC^n \setminus \SA}^\bullet$ is $a_!$-acyclic by Lemma \ref{lemma2}, we conclude that
\begin{equation*}
H^k_c(\BC^n, Rj_* \BR_{\BC^n \setminus \SA}) \simeq H^k(\BC^n, a_!j_* \SE^\bullet_{\BC^n \setminus \SA}) \simeq H^k(\Gamma_c(\BC^n, j_* \SE^\bullet_{\BC^n \setminus \SA})).
\qedhere
\end{equation*}
\end{proof}

Proposition \ref{lemmaresolution} yields a concrete description of $H^n_{< \infty}(\BC^n \setminus \SA)_\BR.$ It consists of cohomology classes of smooth real forms $\omega$ on $\BC^n \setminus \SA$ such that there exists a compact $K \subseteq \BC^n$ for which $\omega$ vanishes outside $K.$ This is notably different from compactly supported cohomology of the complement $H^n_c(\BC^n \setminus \SA)$ where the compact support is required inside $\BC^n \setminus \SA.$

\printbibliography
\end{document}